\documentclass[a4paper,10pt]{article}
\usepackage[utf8]{inputenc}
\usepackage{gensymb}
\usepackage[square,numbers]{natbib}
\usepackage{amssymb}
\usepackage{amsfonts}
\usepackage{latexsym}
\usepackage{amsmath,systeme}
\usepackage{graphics}
\usepackage{subfigure}
\usepackage{graphicx}
\usepackage{epsfig}
\usepackage[leqno]{mathtools}
\usepackage{amsthm}
\usepackage[titletoc]{appendix}
\usepackage[dvipsnames,usenames]{xcolor}
\usepackage{etoolbox}
\usepackage{float}

\mathtoolsset{showonlyrefs}

\numberwithin{equation}{section}

\newtheorem{theorem}{Theorem}[section]

\newtheorem{lemma}[theorem]{Lemma}

\newtheorem{remark}{Remark}

\numberwithin{equation}{section}

\usepackage{esint}

\def\vint_#1{\mathchoice%
          {\mathop{\kern 0.2em\vrule width 0.6em height 0.69678ex depth -0.58065ex
                  \kern -0.8em \intop}\nolimits_{\kern -0.4em#1}}%
          {\mathop{\kern 0.1em\vrule width 0.5em height 0.69678ex depth -0.60387ex
                  \kern -0.6em \intop}\nolimits_{#1}}%
          {\mathop{\kern 0.1em\vrule width 0.5em height 0.69678ex
              depth -0.60387ex
                  \kern -0.6em \intop}\nolimits_{#1}}%
          {\mathop{\kern 0.1em\vrule width 0.5em height 0.69678ex depth -0.60387ex
                  \kern -0.6em \intop}\nolimits_{#1}}}
                  
                  \newcommand{\aveint}[2]{\mathchoice%
          {\mathop{\kern 0.2em\vrule width 0.6em height 0.69678ex depth -0.58065ex
                  \kern -0.8em \intop}\nolimits_{\kern -0.45em#1}^{#2}}%
          {\mathop{\kern 0.1em\vrule width 0.5em height 0.69678ex depth -0.60387ex
                  \kern -0.6em \intop}\nolimits_{#1}^{#2}}%
          {\mathop{\kern 0.1em\vrule width 0.5em height 0.69678ex depth -0.60387ex
                  \kern -0.6em \intop}\nolimits_{#1}^{#2}}%
          {\mathop{\kern 0.1em\vrule width 0.5em height 0.69678ex depth -0.60387ex
                  \kern -0.6em \intop}\nolimits_{#1}^{#2}}}

\begin{document}

%\medskip
\title{\bf Coupled local/nonlocal models in thin domains}

\author{Bruna C. dos Santos, Sergio M. Oliva and Julio D. Rossi}

\maketitle

\begin{abstract} 
In this paper, we analyze a model composed by coupled local and nonlocal diffusion equations
acting in different subdomains.
We consider the limit case when one of the subdomains is thin in one direction (it is
concentrated to a domain of smaller dimension) and 
as a limit problem we obtain coupling between local and nonlocal equations
acting in domains of different dimension. We find existence and uniqueness 
of solutions and we prove several qualitative properties (like conservation of mass and 
convergence to the mean value of the initial condition as time goes to infinity).
\end{abstract}

\noindent{\makebox[1in]\hrulefill}\newline2010 \textit{Mathematics Subject
Classification.} % 35R11, %Fractional partial differential equations
%35S10, %Initial value problems for PsDO
%35B65, %Smoothness and regularity of solutions
35K55, %Nonlinear PDE of parabolic type
35B40, % Asymptotic behavior of solutions
%26A33, %Fractional derivatives and integrals
35A05. %General existence and uniqueness theorems
%35K65, %Parabolic partial differential equations of degenerate type
%76S05 %Flows in porous media; filtration; seepage

%\tableofcontents

\section{Introduction and main results}

In this paper we combine a local diffusion equation, the classical heat equation,
\begin{equation} \label{ec.calor}
\frac{\partial u}{\partial t} (x,t) = \Delta u (x,t)
\end{equation}
in a higher dimensional domain $\Omega \subset \mathbb{R}^N$,
with a nonlocal diffusion equation, given by an integrable kernel
\begin{equation} \label{ec.no-local}
\frac{\partial u}{\partial t} (x,t) = \int_{R} J(x-y) (u(y,t) -u(x,t) ) dy
\end{equation}
in $R$ a different subset of $\mathbb{R}^{N}$. Associated with these two domains, $\Omega$ and $R$,
in \cite{GQR} and \cite{santos2020} the following kind of energy functional was introduced
\begin{equation}\label{neumann-energy.intro}
\begin{array}{l}
    E(u,v)  \displaystyle:=\frac{1}{2}\int_{\Omega} |\nabla  u |^2 dx 
    + \frac{1}{4}\int_{R}\int_{R}J(x-y)\left(v(y)-v(x)\right)^2 dydx  
    + \frac{1}{2}\int_{R}\int_{A}G(x-y)\left(v(x)-u(y)\right)^2 dy dx.
\end{array}
\end{equation}
Here the set $A \subset \Omega$ is the whole $\Omega$ (and we will refer to the resulting model as
having a coupling in the source terms, see the next subsection) or a part of the boundary
$A=\Gamma \subset \partial \Omega$ (we refer to this case as coupling at the boundary).

Observe that, the kernels $J$ and $G$ do not need to be equal. 
We will assume that $J$ and also $G$ satisfy the following hypotheses that will be assumed along the whole paper without further mention,
\begin{flalign*}
 & \text{ $J \in C (\mathbb{R}^{N},\mathbb{R})$ is nonnegative, with $J(0)>0$, $J(-x)=J(x)$ for every $x \in \mathbb{R}^{N}$, and integrable,}  \\
 & \text{ $G \in C(\mathbb{R}^{N},\mathbb{R})$ is nonnegative, nontrivial and integrable.} 
\end{flalign*}

\begin{remark} {\rm 
We can also consider kernels that are not in convolution form, that is, $J(x,y)$ and $G(x,y)$ with
$J \in C (\mathbb{R}^{N} \times \mathbb{R}^{N},\mathbb{R})$ nonnegative, with $J(x,x)>0$, symmetric $J(x,y)=J(y,x)$
and integrable, and $G \in C(\mathbb{R}^{N} \times \mathbb{R}^{N},\mathbb{R})$ nonnegative, nontrivial and integrable.
To simplify the presentation we will deal with convolution type kernels in the proofs. 

Observe that it is common to assume that the integral of $J$ and $G$ is equal to one.
This assumption is related to the probabilistic interpretation of the model given in \cite{GQR} and \cite{santos2020}. For example, in this interpretation, $G((x_1,x_2),y)$ is the probability of a particle (or an individual of a biological species) that is at $(x_1,x_2)$ jumps to $y$ in a time step). So, in this case, we have
\begin{equation*}
\int_{R_{1}} G(x_1,x_2,y) dy = 1.
\end{equation*}
To obtain our results we only need the integrability of the kernels, hence we do not assume that they are normalized to have
integral equal to one.
}
\end{remark}

Associated with the energy \eqref{neumann-energy.intro} we have the evolution problem given by its gradient flow 
(with respect to $L^2(\Omega \cup R)$). 
This gives rise to an diffusion problem. Take
$(u,v)$ as the solution of the abstract $ODE$ problem $$(u,v)'(t)=-\partial E\left[(u,v)(t)\right], \qquad t\geq 0,$$ 
with $u(x,0)=u_{0}(x)$, $v(x,0) =v_0(x)$ and, $\partial E\left[(u,v)\right]$ the subdifferential of $E$. Then, it turns out
(see \cite{GQR} and \cite{santos2020}) that $(u,v)$ solves a system composed by a heat equation (local diffusion)
of the form \eqref{ec.calor} in 
$\Omega$ and a nonlocal difusion equation in $R$, \eqref{ec.no-local}, coupled via source terms 
in the equations (when $A=\Omega$ in \eqref{ec.no-local}) or via a boundary flux on $\Gamma \subset \partial \Omega$
(when $A=\Gamma$ in \eqref{ec.no-local}). See Sections \ref{neumamm-case.source} and \ref{neumamm-case}
 below.

Also from \cite{GQR} and \cite{santos2020} we know that
the associated evolution problem is well-posed in the sense that there are existence and uniqueness of solutions. 
There are two alternative proofs of this fact. 
The first one uses a fixed point argument while the second relies on semigroup theory. Besides, a comparison 
principle holds. Also, the total mass of the initial condition is preserved along the evolution and the solutions converge exponentially fast to the mean value of the initial condition. 
Notice that, according to \cite{GQR} and \cite{santos2020}, we do not impose any continuity of the densities 
 troughout the interface between the local and nonlocal domain, 
but we can guarantee continuity of the densities $u$ and $v$ inside the local and nonlocal subdomains $\Omega$ and $R$, respectively, by assuming continuity of the initial conditions. 
Also there is a probabilistic interpretation of this model (we refer one more time to \cite{GQR} and \cite{santos2020}). 
In this interpretation individuals cannot diffuse neither jump from the exterior $\mathbb{R}^{N}\setminus \Omega$ into $\Omega$ or the other way around (the integrals accounting for jumps do not consider the complement of $\Omega$). 
There is no interchange of mass between $\Omega \cup R$ and its complement. 
Therefore, the total mass is preserved and we can call our problem as being of Neumann type.

The study of nonlocal problems with smooth kernels has been widely considered recently, see
\cite{BCh,Bere,CF,ChChRo,Cortazar-Elgueta-Rossi-Wolanski,delia1,F,FW,Gal} and the book \cite{andreu2010nonlocal}. This kind of equation is getting attention due to its potential applications in ecology, physics, and engineering, and to its flexibility to accurately capture effects that are not easily obtained from classical local models. Biological mobility models of animals and plants are examples of how distinct patterns of mobility can affect the success of invasions \cite{Bere, Strick}. In epidemiology, the effects of long-range interactions are responsible for the spreading of diseases around the world \cite{wang}. Nonlocal patterns also play an important role in molecular interactions in dissimilar interfaces, continuum mechanics, \cite{Han,Sel}, and peridynamics (a model of elasticity and mechanics), \cite{Sil,Sil1}. 

There are different strategies for couplings between local and nonlocal models. Let us briefly summarize previous results in \cite{delia2,Du,Gal,GQR,Kri,santos2020}, see also the review \cite{deliaYY}. In \cite{delia2}, local and nonlocal problems are coupled trough a prescribed solid region in which both kinds of equations overlap (the value of the solution in the nonlocal part of the domain is used as a Dirichlet boundary condition for the local part and vice-versa). This kind of coupling gives continuity of the solution in the overlapping region but does not preserve the total mass. Here we follow \cite{GQR} and \cite{santos2020} (see also \cite{Gal,Kri}). In probabilistic terms, in the model described in \cite{GQR}, particles may jump across the interface between the two regions but can not pass coming from the local side unless they jump.
Finally, in \cite{santos2020}, the authors studied local and nonlocal diffusion models in different zones coupled via the fluxes across the surface that separates the two regions.

Here, we take as the nonlocal region a thin domain, that is, we consider 
$R_{\varepsilon} \subset \mathbb{R}^{N}$ ($R_{\varepsilon}$ is assumed to be open and bounded), depending on a small parameter $\varepsilon \in (0,1]$ that will go to zero and 
that measures the thickness of the domain. Therefore, in our model problem we have two full dimensional domains, 
the local domain $\Omega \subset \mathbb{R}^{N}$ (that is fixed) and the nonlocal domain $R_{\varepsilon} \subset \mathbb{R}^{N} = \mathbb{R}^{N_1}\times \mathbb{R}^{N_2}$. 
We denote $x = (x_1,x_2)$ a point in $\mathbb{R}^N=\mathbb{R}^{N_1}\times \mathbb{R}^{N_2}$. 
The domain $R_{\varepsilon}$ is assumed to be a general thin domain defined as
$$
R_{\varepsilon} = \big\{(x_1,\varepsilon x_2) \in \mathbb{R}^{N_1}\times \mathbb{R}^{N_2} : (x_1,x_2) \in R \big\},
$$
with $R\subset \mathbb{R}^{N} = \mathbb{R}^{N_1}\times \mathbb{R}^{N_2}$. Notice that $R_\varepsilon$ is a 
domain that is thin in the $x_2$-variable. See Figure 1.

Our main goal here is to pass to the limit as $\varepsilon \to 0$ in the previous setting and 
obtain a nontrivial diffusion model in which we couple local and nonlocal diffusion 
equations, \eqref{ec.calor} and \eqref{ec.no-local} that take place in domains of different dimension 
(we deal here with local diffusion in the full-dimensional domain and nonlocal diffusion in the lower-dimensional one).

For simplicity, we will concentrate in the product case and take $R_{\varepsilon}$
as  $$R_{\varepsilon} = R_1 \times \varepsilon R_2 =\{(x_1, \varepsilon x_2): x_1 \in R_1 ,  x_2 \in R_2 \}.$$ Our results are valid in a more
general setting (see Remark \ref{rem.intro} below) but we prefer to avoid extra notations and simplify the
changes of variables that are needed in the proofs. The typical configuration under study is depicted in
Figure 1.

\begin{remark}  \label{rem.intro} {\rm
Instead of a thin domain like $R_{\varepsilon} = \{(x_1, \varepsilon x_2): x_1 \in R_1 ,   x_2 \in R_2 \}$, 
we could have a more complex domain, which could be described by some function $g$ related to the geometry of the channel $R_{\varepsilon}$, 
more exactly, on the way the channel $R_{\varepsilon}$ collapses to a general manifold $R_1$. If we want to construct a more general 
geometry of the channel we could, for instance, in two dimensions, consider the channel $R_{\varepsilon} = \{(x, y): 0 < x_1 < 1, 0 < x_2 < \varepsilon g(x_1)\}$, although more general and complicated geometries are allowed, see \cite{arrieta1}.}
\end{remark}

\begin{figure}[ht!]
\centering
\includegraphics[width=8.5cm]{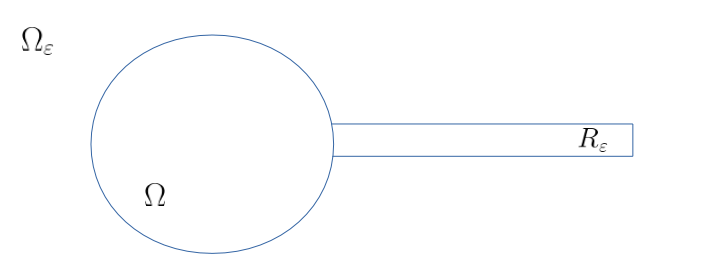}
\caption{Perturbed domain.}
\label{fig1}
\end{figure}

\begin{figure}[ht!]
\centering
\includegraphics[width=8.5cm]{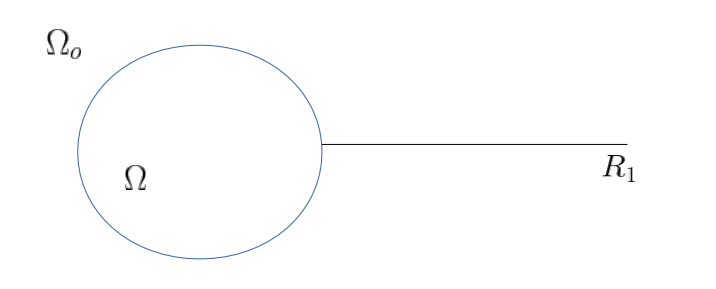}
\caption{Limit domain.}
\label{fig2}
\end{figure}

{\bf Main goal.} Let $\Omega_{\varepsilon}= \Omega \cup R_{\varepsilon} \subset \mathbb{R}^{N}$ and 
consider a local/nonlocal coupling in this domain (see subsections \ref{neumamm-case.source} and 
\ref{neumamm-case} for a precise statement of the involved equations and the obtained results). 
As we have mentioned, our main goal is to study the limit as the nonlocal region, $R_{\varepsilon}$,
gets thinner, that is, to study the limit as $\varepsilon \to 0$. When passing the limit as $\varepsilon \to 0$, the "limit" domain, $\Omega_0$ (see figure \ref{fig2}) will be the union of $\Omega$ and the lower dimensional domain $R_{1}$. 
In the limit of the solutions to our coupled models we will obtain solutions to a local equation in the domain $\Omega$ (with a nonlocal source) 
and a nonlocal equation in a domain of smaller dimension, $R_{1}$. After obtaining the limit equations, we will also prove some qualitative 
properties of this limit problem (like conservation of the total mass and study the asymptotic behaviour of the solutions).

Concerning references for equations in thin domains we refer to \cite{arrieta1,arrieta2,arrieta3,pereira-rossi,arrieta-carvalho,shuichi} that develop some techniques and methods to understand the effects of the geometry of the thin domain on the solutions of elliptic and parabolic singular problems. We can find some applications in elastic beam theories (as torsion and warping functions) \cite{rodriguez}, lubrification \cite{bayada}, fluid flows as ocean dynamics, geophysical fluid dynamics, and fluid flows in cell membranes, see for instance \cite{raugel}.

Our results can be viewed as an extension of \cite{arrieta1} and \cite{pereira-rossi}. In \cite{arrieta1}, the authors 
investigate the dynamics of a local reaction-diffusion equation with homogeneous boundary condition in a dumbbell domain. 
The dumbbell domain is composed by two disconnected regions joined by a thin channel, that depends on a thickness parameter $\varepsilon$ and degenerates to a line segment as the parameter $\varepsilon \to 0$. As part of a series of articles (see \cite{arrieta2,arrieta3}) the authors also 
prove some properties about the continuity of the set of equilibria.
On the other hand, in \cite{pereira-rossi} the authors deal with nonlocal evolution problems with non singular kernels in thin domains
obtaining a limit problem when the thickness of the domain goes to zero, but without considering any coupling
with a local part of the problem. Passing to the limit in these coupling terms is the main contribution of this work.

\subsection{Coupling using source terms}\label{neumamm-case.source}

We need to compare the solutions of the problem posed in the perturbed domain $\Omega_{\varepsilon}= \Omega \cup R_{\varepsilon} \subset \mathbb{R}^{N}$ and the solutions to the limit problem in the limit domain $\Omega_{0}$. Since the solutions live in different spaces, to obtain convergence we need some care, not only 
in the choice of the functional space, but also with the metric chosen in this space.
Decomposing a function $w \in L^2 (\Omega_{\varepsilon})$ as $w =u+v$, with $u =w \chi_{\Omega}$
and $v = w \chi_{R_{\varepsilon}}$, we define the metric in $L^2(\Omega_{\varepsilon})$ as
\begin{equation}\label{different-metric}
\| w \|^2_{L^2 {(\Omega_{\varepsilon})}}=\int_{\Omega} |u|^2 + \frac{1}{\varepsilon^{N_2}} \int_{R_{\varepsilon}}|v|^2.
\end{equation}
Remark that we multiply the norm of the involved functions in the thin part of the domain $R_{\varepsilon}$ 
by a factor $\varepsilon^{-(N_2)}$. 
Now, we can define the energy functional
\begin{equation}\label{neumann-energy.44}
    E_\varepsilon (u,v)  \displaystyle:=\frac{1}{2}\int_{\Omega} |\nabla  u |^2 dx 
    + \frac{1}{4 \varepsilon^{2N_2}}\int_{R_{\varepsilon}}\int_{R_{\varepsilon}}J(x-y) |v(y)-v(x)|^2 dydx  + \frac{1}{2\varepsilon^{N_2}}\int_{R_{\varepsilon}}\int_{\Omega}G(x-y)|v(x)-u(y)|^2 dy dx,
\end{equation}
which is finite in 
$$\mathcal{B}:=\left\{(u,v) \in L^2(\Omega_{\varepsilon}): u \in H^1(\Omega),  v \in L^2(R_{\varepsilon}) \right\}.$$
Notice that in this energy functional we have two terms,
$$
\frac{1}{2}\int_{\Omega} |\nabla  u |^2 dx \qquad \mbox{and} \qquad \frac{1}{4 \varepsilon^{2N_2}}\int_{R_{\varepsilon}}\int_{R_{\varepsilon}}J(x-y)\left(v(y)-v(x)\right)^2 dydx,
$$
that are naturally associated with the equations \eqref{ec.calor} and \eqref{ec.no-local} plus a coupling term given by
$$
\frac{1}{2\varepsilon^{N_2}}\int_{R_{\varepsilon}}\int_{\Omega}G(x-y)\left(v(x)-u(y)\right)^2 dy dx.
$$

Now, let us consider the evolution problem obtained as the gradient flow associated with this energy with respect to the norm previously defined in
\eqref{different-metric}, that is,
$(u (t),v(t))$ will be the solution of the abstract $ODE$ problem $$(u,v)'(t)=-\partial E_\varepsilon \left[(u,v)(t)\right], \qquad t\geq 0,$$ 
with initial data $u(x,0)=u_{0}(x)$, $v(x,0) =v_0^{\varepsilon}(x)$. Here  $\partial E\left[(u,v)\right]$ denotes the subdifferential of $E$ at the point $(u,v)$. 
To see what kind of equations we are solving here, let us compute the derivative of $E$ at $(u,v)$ in the direction of $\varphi \in C_{0}^{\infty}(\Omega_{\varepsilon})$,
\begin{align*}
 \displaystyle 
 \partial_{\varphi}E_\varepsilon (u,v)  & \displaystyle =\lim_{h \to 0}\frac{E_\varepsilon (u+h\varphi,v+h\varphi)-E_\varepsilon (u,v)}{h} \\
 \displaystyle  & = \int_{\Omega} \nabla u \nabla \varphi dx + \frac{1}{\varepsilon^{N_2}}\int_{R_{\varepsilon}}\frac{1}{2 \varepsilon^{N_2}}\int_{R_{\varepsilon}}J(x-y)(v(y)-v(x))(\varphi(y)-\varphi(x))dydx  \\
 \displaystyle  & \qquad + \frac{1}{\varepsilon^{N_2}}\int_{R_{\varepsilon}}\int_{\Omega}G(x-y)(v(x)-u(y))dy \varphi(x)dx -  \int_{\Omega}\frac{1}{\varepsilon^{N_2}}\int_{R_{\varepsilon}}G(x-y)(v(y)-u(x))dy \varphi(x)dx.
\end{align*}
Since $\langle \partial E[u,v],\varphi \rangle=\partial_{\varphi}E(u,v)$, we can derive the local/nonlocal problem associated to this gradient flow
that is given by the following system of equations:
\begin{align}\label{eps-case1}
\begin{cases}
 \displaystyle   \frac{\partial u}{\partial t}(x,t)  = \bigtriangleup u(x,t) + \frac{1}{\varepsilon^{N_2}}\int_{R_{\varepsilon}}G(x-y)(v(y,t)-u(x,t))dy, 
 \quad (x,t) \in \Omega\times (0,+\infty),  \\[10pt]
 \displaystyle   \frac{\partial u}{\partial \eta}(x,t)  =  0, \quad (x,t) \in \partial \Omega\times (0,+\infty),  \\[10pt]
 \displaystyle \frac{\partial v}{\partial t}(x,t)=\frac{1}{\varepsilon^{N_2}}\int_{R_{\varepsilon}} J(x-y)\left(v(y,t)-v(x,t) \right)dy 
   - \int_{\Omega}G(x-y)(v(x,t)-  u(y,t))dy, \ (x,t) \in R_{\varepsilon} \times (0,+\infty),  \\[10pt]
 \displaystyle    u(x,0)  =u_{0}(x), \quad x \in \Omega, \\[5pt]
 v(x,0)  =v_{0}^{\varepsilon}(x), \quad x \in R_{\varepsilon}.
    \end{cases}
    \end{align}
    
    As we have mentioned previously, our aim is to pass to the limit as $\varepsilon \to 0$ in this evolution problem
    \eqref{eps-case1}.
To introduce a candidate to be a limit problem for \eqref{eps-case1}, defined in the domain $\Omega_{0}$ (see Figure \ref{fig2}) 
we will perform a change of variables (as described in \cite{pereira-rossi}) in the thin domain, $R_{\varepsilon}$, in order to fix it. 
The change of variables is given by
$$
R = R_{1} \times R_{2} \ni (x_1,x_2)\longmapsto (x_1, \varepsilon x_2) \in R_1 \times \varepsilon R_2 = R_{\varepsilon}.
$$
That is, we take $\tilde{x}_{2}=\frac{x_2}{\varepsilon}$ and $\tilde{y}_{2}=\frac{y_2}{\varepsilon}$.
With these variables we can fix the domain which allows us to analyze the asymptotic behavior as $\varepsilon \to 0$ in a fixed space of functions. 
To fix the initial condition for $v$ after the change of variables, we take $v(x,0)=v_{0}^{\varepsilon}(x_1,x_2) = v_0 (x_1, \tilde{x}_2)$ for some fixed function $v_0$.
The problem \eqref{eps-case1} becomes after this change of variables the following equations in the fixed domain $\widehat{\Omega} = \Omega \cup R$:
\begin{align}\label{neumann-fixed-domain}
\begin{cases}
 \displaystyle   \frac{\partial u^{\varepsilon}}{\partial t}(x,t)  = \bigtriangleup u^{\varepsilon}(x,t) + \int_{R}G_{\varepsilon}(x-y)(v^{\varepsilon}(\tilde{y},t)-u(x,t))d\tilde{y}, \quad (x,t) \in \Omega \times (0,+\infty),  \\[10pt]
 \displaystyle   \frac{\partial u^{\varepsilon}}{\partial \eta}(x,t)  =  0, \quad (x,t) \in \partial \Omega\times (0,+\infty),  \\[10pt]
 \displaystyle \frac{\partial v^{\varepsilon}}{\partial t}(\tilde{x},t)=\int_{R} J_{\varepsilon}(x-y)\left(v^{\varepsilon}(\tilde{y},t)-v^{\varepsilon}(\tilde{x},t) \right)d\tilde{y} 
   - \int_{\Omega}G_{\varepsilon}(x-y)(v^{\varepsilon}(\tilde{x},t)-  u^{\varepsilon}(y,t))dy, \ ( \tilde{x},t) \in R \times (0,+\infty),  \\[10pt]
 \displaystyle    u^\varepsilon (x,0)  =u_{0}(x), \quad x \in \Omega, \\[5pt]
 v^\varepsilon (\tilde{x},0)  =v_{0}(\tilde{x}), \quad \tilde{x} \in R,
    \end{cases}
    \end{align}
where $$J_{\varepsilon}(x-y)=J(x_1-y_1,\varepsilon(\tilde{x}_2-\tilde{y}_2)),  \quad G_{\varepsilon}(x-y)=G(x_1-y_1,\varepsilon \tilde{x}_2-\tilde{y}_2)
\quad \mbox{and} \quad v^{\varepsilon}(x_1,\tilde{x}_2,t)=v(x_1,\varepsilon \tilde{x}_2,t).$$
Notice that the problem \eqref{neumann-fixed-domain} is similar to the ones obtained previously in thin domains (see for instance \cite{pereira-rossi,arrieta1}). 

Now, we are ready to state our main result for this coupling.

\begin{theorem}\label{teo1}
Let $\{ (u^{\varepsilon},v^{\varepsilon}) \}_{\varepsilon >0}$ be a family of solutions of \eqref{neumann-fixed-domain}. Then, there exist $(u^{*},V^{*})$, $u^{*} \in C([0,T],H^{1}(\Omega))$ and $V^{*} \in C([0,T],L^{2}(R_1))$, such that 
$$
\begin{array}{l}
\displaystyle
u^{\varepsilon} \rightharpoonup u^{*} \quad \mbox{in} \quad L^{\infty} (0,T;L^2(\Omega)), \\
\displaystyle v^{\varepsilon} \rightharpoonup v^{*} \quad \mbox{in} \quad L^{\infty} (0,T;L^2(R)) \quad \mbox{and}, \\
\displaystyle V^{\varepsilon}(\cdot) = \int_{R_2}v^{\varepsilon}(\cdot,\varepsilon \tilde{x}_2,t)d\tilde{x}_2 \rightharpoonup 
V^{*}(\cdot) = \int_{R_2}v^{*}(\cdot,0,t)d\tilde{x}_2 \quad \mbox{in} \quad L^{\infty} (0,T;L^2(R_1)).
\end{array}
$$
The pair $\{u^{*},V^{*}\}$ satisfies the following limit problem in $\Omega_0 = \Omega \cup R_1$,
 \begin{align}\label{limit-problem-neumann}
\begin{cases}
 \displaystyle   \frac{\partial u^{*}}{\partial t}(x,t)  =  \Delta u^{*} (x,t)+\int_{R_1} G^{*}(x-y)(V^{*}(y,t)-|R_2|u^{*}(x,t))dy, 
 \quad (x,t) \in \Omega \times (0,+\infty),\\[10pt]
 \displaystyle   \frac{\partial u^{*}}{\partial \eta}(x,t)  =  0, \quad (x,t) \in \partial \Omega \times (0,+\infty),  \\[10pt]
 \displaystyle \frac{\partial V^{*}}{\partial t}(x,t)  = |R_2|\int_{R_1} J^{*}(x-y)\left(V^{*}(y,t)-V^{*}(x,t) \right)dy 
   - \int_\Omega G^{*}(x-y)(V^{*}(x,t)- |R_2| u^{*}(y,t)) dy,  \\[10pt]
   \displaystyle \qquad\qquad\qquad\qquad\qquad\qquad\qquad\qquad\quad (x,t) \in R_1 \times (0,+\infty)  \\[10pt]
     u^{*}(x,0)  =u_{0} (x), \quad x \in \Omega,\\[5pt]
 \displaystyle    V^{*}(x,0)=V_{0}^{*}(x)= \int_{R_2}v_0(x,0)dx_2, \quad x \in R_1,
    \end{cases}
    \end{align}  
where the limit kernels $J^{*}$ and $G^{*}$ are given by
$$
J^{*}(x-y) = J(x_1-y_1,0), \qquad
G^{*}(x-y) = G(x_1-y_1,0-y_2).
$$
\end{theorem} 

We also include here some properties of the limit problem \eqref{limit-problem-neumann}. 
The problem is well posed, the total mass remains constant in time, that is,
\begin{equation}\label{mass.intro}
\int_{\Omega}u^{*}(x,t)dx+\int_{R_1}V^{*}(x,t)dx = \int_{\Omega}u^{*}_0 dx+\int_{R_1}V^{*}_0 dx, \quad \forall t \geq 0,
\end{equation} 
and solutions converge exponentially to the mean value of the initial condition as $t\to \infty$, i.e.,
$$
\Big\| (u^{*},V^*)(\cdot,t)-\fint (u_{0}^{*},V^*_0) \Big\|_{L^2({\Omega}_0)} \leq C 
 e ^{- \lambda_{1} t}.
$$
for some $C>0$ and $\lambda_1>0$ (we also obtain that $\lambda_1$ can be chosen independent of the initial data).

\subsection{Coupling at the boundary}\label{neumamm-case}

Now we want to impose that an individual to pass from the nonlocal domain to the local domain, it necessarily needs to cross the boundary to then get in the local domain.  
As we did before, first we will define the problem in the perturbed domain (see Figure \ref{fig1}) and then 
derive the limit problem defined in the limit domain (see Figure \ref{fig2}). 
Let us, as before, consider the domain $\Omega_{\varepsilon}= \Omega \cup R_{\varepsilon} \subset \mathbb{R}^{N}$, with $\Omega \subset \mathbb{R}^{N}$ and $R_{\varepsilon} = R_1 \times \varepsilon R_2 \subset \mathbb{R}^{N}$, with a small parameter $\varepsilon$.

Let us consider the metric \eqref{different-metric} and derive the evolution problem as the flux associated with the energy
\begin{equation} \label{energy2}
    E^b_\varepsilon (u,v)  \displaystyle:=\frac{1}{2}\int_{\Omega} |\nabla  u |^2 dx 
    + \frac{1}{4\varepsilon ^{2 N_{2}}}\int_{R_\varepsilon}\int_{R_\varepsilon}J(x-y)\left(v(y)-v(x)\right)^2 dydx 
    \displaystyle  + \frac{1}{2\varepsilon^{N_2}}\int_{R_{\varepsilon}}\int_{\Gamma}G(x-y)\left(v(x)-u(y)\right)^2 d\sigma(y) dx,
\end{equation}
with $\Gamma$ a fixed part of the boundary of $\Omega$ with $|\Gamma|_{N-1}>0$ (then we have a well defined trace
operator from $H^1(\Omega)$ into $L^2(\Gamma))$. 
Notice that the coupling term 
$$
\frac{1}{2\varepsilon^{N_2}}\int_{R_{\varepsilon}}\int_{\Gamma}G(x-y)\left(v(x)-u(y)\right)^2 d\sigma(y) dx
$$
involve the values of $u$ on $\Gamma \subset \partial \Omega$ instead of the values of $u$ inside $\Omega$ (compare with 
the previous functional $E_\varepsilon (u,v)$).

Now, the evolution problem associated to the energy functional $E^b_\varepsilon (u,v)$ is given by the following system:
\begin{align}\label{problem-energy2}
\begin{cases}
 \displaystyle   \frac{\partial u}{\partial t}(x,t)  =  \Delta u (x,t), \quad (x,t) \in \Omega \times (0,+\infty),\\[10pt]
 \displaystyle   \frac{\partial u}{\partial \eta}(x,t)  =  0, \quad (x,t) \in \partial \Omega \setminus \Gamma \times (0,+\infty),  \\[10pt]
 \displaystyle     \frac{\partial u}{\partial \eta}(x,t)  = \frac{1}{\varepsilon^ {N_2}}\int_{R_{\varepsilon}} G(x-y)( v(y,t)-  u(x,t))d\sigma(y), \quad (x,t) \in  \Gamma \times  (0,+\infty), \\[10pt] 
 \displaystyle \frac{\partial v}{\partial t}(x,t)  = \frac{1}{\varepsilon^{N_2}}\int_{R_{\varepsilon}} J(x-y)\left(v(y,t)-v(x,t) \right)dy 
   - \int_\Gamma G(x-y)   ( v(x,t)-  u(y,t)) d\sigma(y), \ (x,t) \in R_{\varepsilon} \times (0,+\infty)  \\[10pt]
     u(x,0)  =u_{0}(x), \quad x \in \Omega,\\[5pt]
 \displaystyle    v(x,0)  =v_{0}^{\varepsilon}(x), \quad x \in R_{\varepsilon}.
    \end{cases}
    \end{align}
Notice that the nonlocal part contributes with the normal derivative of $u$ on $\Gamma$ and the local part of the problem appears as before 
in the source term of the equation for the nonlocal part. The coupling is balanced in such a way that the problem preserves the total mass,
see \cite{santos2020}.

After the same change of variables that we used before, 
$\tilde{x}_{2}=\frac{x_2}{\varepsilon}$ and $\tilde{y}_{2}=\frac{y_2}{\varepsilon}$,
we fix the domain and then pass to the limit and obtain the limit problem.
Again here we take $v(x,0)=v_{0}^{\varepsilon}(x_1,x_2) = v_0 (x_1, \tilde{x}_2)$ for some fixed function $v_0$ as the initial condition.
Notice that, as we did in the previous subsection, there exists an equivalence between the coupled local/nonlocal problem \eqref{problem-energy2} with the following coupled local/nonlocal thin domain problem defined in $\hat{\Omega} = \Omega \cup R$, with $\Gamma$ a fixed 
part of the boundary of $\Omega$,
\begin{align}\label{problemrescaled-energy2}
\begin{cases}
 \displaystyle   \frac{\partial u^{\varepsilon}}{\partial t}(x,t)  =  \Delta u^{\varepsilon} (x,t), \quad (x,t) \in \Omega \times (0,+\infty),\\[10pt]
 \displaystyle   \frac{\partial u^{\varepsilon}}{\partial \eta}(x,t)  =  0, \quad (x,t) \in \partial \Omega \setminus \Gamma \times (0,+\infty),  \\[10pt]
 \displaystyle     \frac{\partial u^{\varepsilon}}{\partial \eta}(x,t)  = \int_{R} G^{\varepsilon}(x-y)( v^{\varepsilon}(\tilde{y},t)-  u^{\varepsilon}(x,t))d\tilde{y}, \quad (x,t) \in  \Gamma \times  (0,+\infty), \\[10pt] 
 \displaystyle \frac{\partial v^{\varepsilon}}{\partial t}(\tilde{x},t)  = \int_{R} J^{\varepsilon}(x-y)\left(v^{\varepsilon}(\tilde{y},t)-v^{\varepsilon}(\tilde{x},t) \right)d\tilde{y} 
   - \int_\Gamma G^{\varepsilon}(x-y)( v^{\varepsilon}(\tilde{x},t)-  u^{\varepsilon}(y,t)) d\sigma (\tilde{y}), \ (\tilde{x},t) \in R \times (0,+\infty)  \\[10pt]
     u^{\varepsilon}(x,0)  =u_{0}(x), \quad x \in \Omega,\\[10pt]
 \displaystyle    v^{\varepsilon}(\tilde{x},0)  =v_{0}(\tilde{x}), \quad \tilde{x} \in R,
    \end{cases}
    \end{align}  
with
$$
J^{\varepsilon}(x-y) = J(x_1-y_1,\varepsilon(\tilde{x}_2-\tilde{y}_2)) \quad 
G^{\varepsilon}(x-y) = G(x_1-y_1,x_2-\varepsilon \tilde{y}_2) \quad \mbox{and} \quad
v^{\varepsilon}(\tilde{x},t)=v(x_1,\varepsilon \tilde{x}_2,t).
$$

Now we can enunciate a convergence result analogous to Theorem \ref{teo1}. It says that there is a limit as $\varepsilon \to 0$ of the solutions to the problem \eqref{problemrescaled-energy2} in the limit domain $\Omega_{0} = \Omega \cup R_{1}$ (see Figure \ref{fig2}).

\begin{theorem}\label{teo2}
Let $\{u^{\varepsilon},v^{\varepsilon}\}_{\varepsilon >0}$ be a family of solutions for the problem \eqref{problemrescaled-energy2}. Then, there exists a solution $(u^{*},V^{*})$, $u^{*} \in C([0,T],H^{1}(\Omega))$ and $V^{*} \in C([0,T],L^{2}(R_1))$, such that 
$$
\begin{array}{l}
\displaystyle
u^{\varepsilon} \rightharpoonup u^{*} \quad \mbox{in} \quad L^{\infty} (0,T;L^2(\Omega)), \\
\displaystyle v^{\varepsilon} \rightharpoonup v^{*} \quad \mbox{in} \quad L^{\infty} (0,T;L^2(R)) \quad \mbox{and}, \\
\displaystyle V^{\varepsilon}(\cdot) = \int_{R_2}v^{\varepsilon}(\cdot,\varepsilon \tilde{x}_2,t)d\tilde{x}_2 \rightharpoonup 
V^{*}(\cdot) = \int_{R_2}v^{*}(\cdot,0,t)d\tilde{x}_2 \quad \mbox{in} \quad L^{\infty} (0,T;L^2(R_1)).
\end{array}
$$
The pair $\{u^{*},V^{*}\}$ satisfies the following limit problem in $\Omega_0=\Omega \times R_1$,
 \begin{align}\label{limit-prob-neumann/robin}
\begin{cases}
 \displaystyle   \frac{\partial u^{*}}{\partial t}(x,t)  =  \Delta u^{*} (x,t), \quad (x,t) \in \Omega \times (0,+\infty),\\[10pt]
 \displaystyle   \frac{\partial u^{*}}{\partial \eta}(x,t)  =  0, \quad (x,t) \in \partial \Omega \setminus \Gamma \times (0,+\infty),  \\[10pt]
 \displaystyle     \frac{\partial u^{*}}{\partial \eta}(x,t)  = \int_{R_1} G^{*}(x-y)(V^{*}(y,t)-|R_2|u^{*}(x,t))dy, \quad (x,t) \in  \Gamma \times  (0,+\infty), \\[10pt] 
 \displaystyle \frac{\partial V^{*}}{\partial t}(x,t)  = |R_2|\int_{R_1} J^{*}(x-y)\left(V^{*}(y,t)-V^{*}(x,t) \right)dy 
   - \int_\Gamma G^{*}(x-y)(V^{*}(x,t)- |R_2| u^{*}(y,t)) d\sigma (y), \\[10pt]
   \qquad \qquad \qquad \qquad \qquad \qquad \qquad \quad (x,t) \in R_1 \times (0,+\infty)  \\[10pt]
     u^{*}(x,0)  =u_{0}^{*}(x), \quad x \in \Omega,\\[5pt]
 \displaystyle    V^{*}(x,0)=V_{0}^{*}(x)= \int_{R_2}v_0(x,0)dx_2, \quad x \in R_1,
    \end{cases}
    \end{align}  
where the limit kernels $J^{*}$ and $G^{*}$ are given by
$$
J^{*}(x-y) = J(x_1-y_1,0), \quad
G^{*}(x-y) = G(x_1-y_1,x_2-0).
$$
\end{theorem}

For this limit problem we also have that it is well posed, the total mass remains constant in time
and solutions converge exponentially to the mean value of the initial condition as $t\to \infty$.

\subsection{The local part in a thin domain}

We can also consider the case in which the local part of the problem takes place in a thin domain (fixing the nonlocal 
domain). That is, we consider $\Omega_1 \subset \mathbb{R}^{N_1}$, $\Omega_2 \subset \mathbb{R}^{N_2}$, and
can take $\Omega_\varepsilon = \Omega_1 \times \varepsilon \Omega_2 \cup R$ as our reference domain. 
In this case the associated energy takes the form 
\begin{equation}\label{neumann-energy.intro.89}
\begin{array}{l}
    E(u,v)  \displaystyle:=\frac{1}{2\varepsilon^{N_2}}\int_{\Omega_1\times \varepsilon \Omega_2} |\nabla  u |^2 dx 
    + \frac{1}{4}\int_{R}\int_{R}J(x-y)\left(v(y)-v(x)\right)^2 dydx  \\[10pt]
    \displaystyle  
    \qquad \qquad + \frac{1}{2\varepsilon^{N_2}}\int_{R}\int_{\Omega_1\times \varepsilon \Omega_2}G(x-y)\left(v(x)-u(y)\right)^2 dy dx.
\end{array}
\end{equation}
We can also consider the limit as $\varepsilon \to 0$ of solutions to the associated gradient flow in this case. 
In this case
we obtain a limit problem in which the equation for $u$ involves only the Laplacian in the first $N_1$-variables 
and the coupling kernel is given by
$$
G^*(x-y) = G(x_1-y_1,0-y_2)
$$
(the kernel $J$ remains unchanged since we are fixing the nonlocal domain $R$). 
The proof of this limit can be obtained following \cite{arrieta1,arrieta2,arrieta3} 
(notice that here we are taking the limit in the local part of the problem) and hence we
don't include the details in this paper.

\medskip

The paper is organized as follows: in Section \ref{NP} we deal with the problem
with coupling via source terms and we prove Theorem \ref{teo1}; in Section \ref{Sect-3}
we consider the coupling on the boundary and prove of Theorem \ref{teo2}; finally, in Section
\ref{sect-numer} we include some numerical experiments (based on a discretization of our
models) that illustrate the behaviour of the solutions to our limit equations.

\section{Coupling via source terms. Proof of Theorem \ref{teo1}}\label{NP}

First, we introduce a result that will be important to study the large time behavior and the limit problems described in the previous Section. 
We state the lemma for the first problem (coupling via source terms) but the same proof can be adapted for the other evolution problem
(coupling on the boundary).

Let us denote by $\hat{\Omega} = \Omega \cup R$ the fixed domain after the change of variables and by $E$ the functional \eqref{neumann-energy.44} after the change of variables, that is, 
\begin{equation}\label{neumann-energy.88}
    E (u,v)  \displaystyle:=\frac{1}{2}\int_{\Omega} |\nabla  u |^2 dx 
    + \frac{1}{4}\int_{R}\int_{R}J_\varepsilon (x-y) |v(y)-v(x)|^2 dydx  + \frac{1}{2}\int_{R}\int_{\Omega}G_\varepsilon (x-y)|v(x)-u(y)|^2 dy dx,
\end{equation}
with
$$
J_{\varepsilon}(x-y) = J(x_1-y_1,\varepsilon(\tilde{x}_2-\tilde{y}_2)),
\qquad
G_{\varepsilon}(x-y) = G(x_1-y_1,\varepsilon\tilde{x}_2-\tilde{y}_2).
$$

\begin{lemma}\label{positive-constant}
Let $\{\lambda_{1}^{\varepsilon}\}_{\varepsilon >0}$ be a family of first nontrivial eigenvalues of our evolution problem that are given by
$$
\lambda_{1}^{\varepsilon} = \inf_{(u,v): \int_{\Omega}u^{\varepsilon}+\int_{R}v^{\varepsilon}=0} \displaystyle 
\frac{E(u^{\varepsilon},v^{\varepsilon})}{\displaystyle \int_\Omega (u^{\varepsilon})^{2}+\int_{R}(v^{\varepsilon})^2}
$$
Then, there exists a constant $C>0$, that does not depends on $\varepsilon$ such that
$$
\lambda_{1}^{\varepsilon}\geq C >0,
$$
and hence we have,
\begin{equation}\label{ineq1}
E(u^{\varepsilon},v^{\varepsilon}) \geq C \left(\int_{\Omega}(u^{\varepsilon})^{2}+\int_{R}(v^{\varepsilon})^2\right),
\end{equation}
for every $(u^{\varepsilon},v^{\varepsilon})$ solution to \eqref{neumann-fixed-domain}, 
such that $\int_{\Omega}u^{\varepsilon}+\int_{R}v^{\varepsilon}=0$.
\end{lemma}

\begin{proof}
Let us argue by contradiction. Suppose that \eqref{ineq1} is not hold, that means that, for every $n \in \mathbb{N}$ there exists a subsequence $\{\varepsilon_n\} \to 0$ and $\{w^{\varepsilon_n}\} = \{(u^{\varepsilon_n}, v^{\varepsilon_n}) \} \in L^{2}(\hat{\Omega})\cap H^{1}(\Omega)$ such that
$$
 \int_{\Omega}u^{\varepsilon_n}+\int_{R}v^{\varepsilon_n}=0,
$$
$$
 \int_{\Omega}(u^{\varepsilon_n})^2+\int_{R}(v^{\varepsilon_n})^2=1,
$$
and
$$
\frac{1}{2}\int_{\Omega}|\nabla u^{\varepsilon_n}|^2 dx+\frac{1}{4}\int_{R}\int_R J_{\varepsilon_n}(x-y)(v^{\varepsilon_n}(y)-v^{\varepsilon_n}(x))^2dydx+\frac{1}{2}\int_{R}\int_{\Omega}G_{\varepsilon_n}(x-y)(v^{\varepsilon_n}(x)-u^{\varepsilon_n}(y))^2dydx \leq \frac{1}{n}.
$$
Taking the limit as $n\to \infty$ we obtain
$$
\lim_{n \to \infty}\left(\frac{1}{2}\int_{\Omega}|\nabla u^{\varepsilon_n}|^2 dx\right)=0,
$$
$$
\lim_{n \to \infty}\left(\frac{1}{4}\int_{R}\int_R J_{\varepsilon_n}(x-y)(v^{\varepsilon_n}(y)-v^{\varepsilon_n}(x))^2dydx\right)=0,
$$
and
$$
\lim_{n \to \infty}\left(\frac{1}{2}\int_{R}\int_{\Omega}G_{\varepsilon_n}(x-y)(v^{\varepsilon_n}(x)-u^{\varepsilon_n}(y))^2dydx\right)=0.
$$

We have that $\int_{\Omega}(u^{\varepsilon_n})^2 dx \leq 1$, that is, $\{u^{\varepsilon_n}\}$ is bounded in $L^2(\Omega)$. 
Moreover, we get that $\{u^{\varepsilon_n}\}$ is bounded in $H^1(\Omega)$. Taking a subsequence, also denoted by $\{u^{\varepsilon_n}\}$, such that $\varepsilon_n \to 0$ we have
$$
u^{\varepsilon_n} \rightharpoonup u^{*} \quad \mbox{ in } H^{1}(\Omega)
$$
$$
u^{\varepsilon_n} \longrightarrow u^{*} \quad \mbox{ in } L^2(\Omega).
$$
Thanks to the Fatou's lemma we know that
$$
\frac{1}{2}\int_{\Omega}|\nabla u^{*}|^2 dx \leq \liminf_{\varepsilon_n}\frac{1}{2} \int_{\Omega}|\nabla u^{\varepsilon_n}|^2 dx = 0.
$$
Hence, the limit $u^{*}$ is constant in $\Omega$.

Also $\{v^{\varepsilon_n}\}$ is bounded in $L^{2}(R)$. Define $k^{\varepsilon_n} = \int_{R}v^{\varepsilon_n}$. From the bound in 
$L^2(R)$ of $v^{\varepsilon_n}$  we obtain that there exists a constant $C$ such that $|k^{\varepsilon_n}|\leq C$ and, moreover, we can take a subsequence $\{v^{{\varepsilon_n}_j}\}$ which weakly converges in $L^2(R)$ to some limit $v^{*}$ as ${\varepsilon_n}_j \to 0$ and 
such that $\{k^{{\varepsilon_n}_j}\}$ also converges to a limit that we call $k^{*}$. Consider $z^{{\varepsilon_n}_j} = v^{{\varepsilon_n}_j}-k^{{\varepsilon_n}_j}$. We have that $\int_{R}z^{{\varepsilon_n}_j}=0$, therefore, see \cite{CERW} and \cite{andreu2010nonlocal}, there exists a constant $C>0$ independent of $\varepsilon$ such that
\begin{equation}\label{julio-ineq}
\int_{R}\int_{R}J(x_1-y_1,{\varepsilon_n}_j(\tilde{x}_2-\tilde{y}_2))(z^{{\varepsilon_n}_j}(y)-z^{{\varepsilon_n}_j}(x))^2d\tilde{y}d\tilde{x} \geq C \int_{R}(z^{{\varepsilon_n}_j}(x))^2 dx.
\end{equation}
In fact, since $J$ is continuous, from our hypothesis on $J$, we get that there exists constants $M, \delta >0$ such that
$$
J(x_1-y_1,x_2-y_2)\geq M, \quad \mbox{whenever} \quad |(x_1-y_1,x_2-y_2)| <\delta.
$$
Then, it follows that
$$
J(x_1-y_1,\varepsilon(x_2-y_2))\geq \frac{M}{2}, \quad \mbox{whenever} \quad |x_1-y_1| <\frac{\delta}{2}, \quad \varepsilon|x_2-y_2|<\frac{\delta}{2},
$$
for every $\varepsilon$ small enough.
Hence, the inequality \eqref{julio-ineq} follows from Lemma 3.1 in \cite{CERW} and the constant $C$ only depends on $M, \delta$ and $R$
but not on $\varepsilon$.

Note that we have
\begin{align*}
\lim_{n \to \infty}\left(\frac{1}{4}\int_{R}\int_R J_{\varepsilon_n}(x-y)(z^{{\varepsilon_n}_j}(y)-z^{{\varepsilon_n}_j}(x))^2dydx\right)=\lim_{n \to \infty}\left(\frac{1}{4}\int_{R}\int_R J_{\varepsilon_n}(x-y)(v^{{\varepsilon_n}_j}(y)-v^{{\varepsilon_n}_j}(x))^2dydx\right) = 0,
\end{align*}
as $\varepsilon_{n_j} \to 0$, 
which yields
$$
0 \geq \lim_{n \to \infty} C \int_{R}(z^{{\varepsilon_n}_j}(x))^2 dx.
$$
From here we conclude that $z^{{\varepsilon_n}_j} \to 0$ in $L^2(R)$, which leads to $v^{{\varepsilon_n}_j} \to k^{*}$ strongly in $L^2(R)$.
Finally, as $u^{\varepsilon_n} \to u^{*}$ in $L^2(\Omega)$ and $v^{\varepsilon_n} \to k^{*}$ in $L^2(R)$, we can take the limit as $\varepsilon_n \to 0$ 
and obtain
\begin{align*}
0=\lim_{n \to \infty}\left(\frac{1}{2}\int_{R}\int_{\Omega}G_{\varepsilon_n}(x-y)(v^{\varepsilon_n}(x)-u^{\varepsilon_n}(y))^2dydx\right)=\frac{1}{2}\int_{R}\int_{\Omega}G^{*}(x-y)(v^{*}(x)-u^{*}(y))^2dydx.
\end{align*}
From where it follows that $k^{*}-u^{*}=0$, that is, $k^{*}=u^{*}$. 
From 
$$
 \int_{\Omega}u^{\varepsilon_n}+\int_{R}v^{\varepsilon_n}=0,
$$
it follows that 
$$
 \int_{\Omega}u^{*}+\int_{R}k^{*}=0,
$$
and since we have $k^{*}=u^{*}$ we get
$$k^{*}=u^{*}=0.$$ 
Now, from 
$$
 \int_{\Omega}(u^{\varepsilon_n})^2+\int_{R}(v^{\varepsilon_n})^2=1,
$$
and the strong convergence in $L^2$ we obtain
$$
 \int_{\Omega}(u^{*})^2+\int_{R}(k^{*})^2=1,
$$
which 
yields a contradiction. The proof is complete.
\end{proof}

With this lemma, following \cite{GQR} (see also \cite{santos2020}), 
we can provide an estimate for the asymptotic behavior of the solutions of the problem \eqref{neumann-fixed-domain}, that is, the solutions 
$\{u^{\varepsilon},v^{\varepsilon}\}_{\varepsilon >0}$ converges to the mean value of the initial condition 
\begin{equation} \label{decay.exp}
 \Big\| (u^{\varepsilon},v^{\varepsilon})(\cdot,t)-\fint (u_0,v_0) \Big\|_{L^2(\widehat{\Omega})} \leq C_{1} e ^{-C_2 t},
\end{equation}
with $C_{1}, C_{2}$ finite positive constants, independent of $\varepsilon$ and also, $C_2$ independent of the initial condition. 
Hence, we have that the $L^2$-norm of $\{u^{\varepsilon},v^{\varepsilon}\}_{\varepsilon >0}$ is bounded (independently of $\varepsilon$).
Here
$$
\fint (u_0,v_0) = \frac{\displaystyle \int_\Omega u_0 + \int_{R} v_0 }{ |\Omega| + |R|}.
$$

Now we are ready to proceed with the proof of Theorem \ref{teo1}.

\begin{proof}[Proof of Theorem \ref{teo1}]
First, we observe that, since $J$ and $G$ are continuous functions, we have
$$
J_{\varepsilon}(x-y) = J(x_1-y_1,\varepsilon(x_2-y_2)) \longrightarrow J^{*}(x-y) = J(x_1-y_1,0), \quad \mbox{and}
$$
$$
G_{\varepsilon}(x-y) = G(x_1-y_1,\varepsilon x_2-y_2) \longrightarrow G^{*}(x-y) = G(x_1-y_1,0-y_2),
$$
as $\varepsilon \to 0$, uniformly in $x,y$.

From Lemma \ref{positive-constant}, since
$\{v^{\varepsilon}\}$
is bounded in $L^\infty (0,T;L^2(R))$ we can take a subsequence, also denoted by $\{v^{\varepsilon}\}$, such that  
$$
v^{\varepsilon} \rightharpoonup v^{*} \quad \mbox{weakly} \quad \mbox{in} \quad L^\infty(0,T;L^2(R)) \quad \mbox{as} \quad \varepsilon \to 0.
$$

On the other hand, we have that
\begin{subequations}
\begin{align}
\int_{\Omega}|u^{\varepsilon}(x,t)|^2dx \qquad \mbox{and} \qquad \int_{\Omega}|\nabla u^{\varepsilon}(x,t)|^2dx, \label{eq:v}
\end{align}
\end{subequations}
are also bounded in $L^2(\Omega)$ (uniformly in $t\in [0,T]$). Hence, along a subsequence if necessary,
\begin{subequations}
\begin{align}
u^{\varepsilon} \rightharpoonup u^{*} \quad \mbox{weakly} \quad \mbox{in} \quad L^\infty(0,T;H^1(\Omega)) \quad \mbox{as} \quad \varepsilon \to 0, \label{eq:u} \\[10pt]
u^{\varepsilon} \to u^{*} \quad \mbox{strongly} \quad \mbox{in} \quad L^\infty(0,T;L^2(\Omega)) \quad \mbox{as} \quad \varepsilon \to 0.
\end{align}
\end{subequations}

Now we consider the weak form of \eqref{neumann-fixed-domain}, that is,
using the symmetry of the kernel $J$ we have the following identities,
\begin{subequations}
\begin{align}
  \int_{\Omega}  u^{\varepsilon}(x,T) \varphi(x,T) dx & - \int_0^T \int_\Omega u^\varepsilon (x,t)\frac{\partial \varphi}{\partial t}(x,t) dxdt  
  = \int_\Omega u_0(x) \varphi (x,0) dx-
  \int_0^T\int_{\Omega} \nabla u^{\varepsilon}(x,t)\nabla \varphi(x,t) dx dt  \\
 & \quad  +\int_0^T \int_{\Omega} \int_{R}G_{\varepsilon}(x-y)(v^{\varepsilon}(\tilde{y},t)-u^{\varepsilon}(x,t))\varphi(x,t) d\tilde{y} dx dt, \label{eq:3} \\
 \int_{R}  v^{\varepsilon}(x,T) \varphi(x,T) dx  & - \int_0^T \int_R v^\varepsilon (x,t) \frac{\partial \varphi}{\partial t} (x,t) dxdt  
 = \int_R v_0(x) \varphi (x,0) dx  \\
 & \quad -\frac{1}{2}\int_0^T
 \int_{R} \int_{R} J_{\varepsilon}(x-y)\left(v^{\varepsilon}(\tilde{y},t)-v^{\varepsilon}(\tilde{x},t) \right)(\varphi(y,t)-\varphi(x,t)) d\tilde{y} dx dt \\
 & \quad - \int_0^T \int_{R} \int_{\Omega}G_{\varepsilon}(x-y)(v^{\varepsilon}(\tilde{x},t)-  u^{\varepsilon}(y,t))\varphi(x,t) dy dx dt,\label{eq:4}  
\end{align}
\end{subequations}
for every $\varphi \in C^1(H^1(\Omega) \cup L^2 (R))$.

Now, let us take a test function that depends only on the first variable, for $x \in R$, that is, $\varphi = \varphi(x_1)$
and us analyze the limit as $\varepsilon \to 0$ of each term in the previous equations. We have
\begin{align}\label{limit-purelocal}
\lim_{\varepsilon \longrightarrow 0 }\left( \int_0^T\int_{\Omega} \nabla u^{\varepsilon} \nabla \varphi dx dt\right) = \left( \int_0^T\int_{\Omega} \nabla u^{*} \nabla \varphi dx dt\right).
\end{align}
Now, note that
$$
\begin{array}{l}
\displaystyle
\int_0^T \int_{\Omega} \int_{R_1}\int_{R_2}G(x_1-y_1,x_2-\varepsilon \tilde{y}_2)(v^{\varepsilon}(y_1,\tilde{y_2},t)-u^{\varepsilon}(x_1,x_2,t))\varphi(x_1,x_2,t) d\tilde{y_2}dy_1dx_2dx_1dt \\[10pt]
\displaystyle
=\int_0^T \int_{\Omega} \int_{R_1}\int_{R_2} \Big[ G (x_1-y_1,x_2-\varepsilon \tilde{y}_2) - G (x_1-y_1,x_2)  \Big]
\\[10pt]
\qquad \qquad \qquad \qquad \qquad \qquad \qquad  \displaystyle \times (v^{\varepsilon}(y_1,\tilde{y_2},t)-u^{\varepsilon}(x_1,x_2,t))\varphi(x_1,x_2,t) d\tilde{y_2}dy_1dx_2dx_1dt \\[10pt]
\quad \displaystyle 
+\int_0^T \int_{\Omega} \int_{R_1}\int_{R_2}G(x_1-y_1,x_2)(v^{\varepsilon}(y_1,\tilde{y_2},t)-u^{\varepsilon}(x_1,x_2,t))\varphi(x_1,x_2,t) d\tilde{y_2}dy_1dx_2dx_1dt.
\end{array}
$$
Notice that the measure in $\Omega$ is the product measure and hence when we integrate we have $dx= dx_1 dx_2$.

Since 
$$
 \Big[ G (x_1-y_1,x_2-\varepsilon \tilde{y}_2) - G(x_1-y_1,x_2)  \Big]
$$
goes to zero uniformly and $u^\varepsilon$ and $v^\varepsilon$ are bounded in $L^2$, the first term goes to zero as $\varepsilon \to 0$
and therefore we concentrate in the second. 
To analyze the limit of the second term, we observe that $G(x_1-y_1,x_2)$ does not depend on $y_2$ and hence we can rewrite
this term as follows,
$$
\begin{array}{l}
\displaystyle
\int_0^T \int_{\Omega} \int_{R_1}\int_{R_2}G(x_1-y_1,x_2)(v^{\varepsilon}(y_1,\tilde{y_2},t)-u^{\varepsilon}(x_1,x_2,t))\varphi(x_1,x_2,t) d\tilde{y_2}dy_1dx_2dx_1dt \\[10pt]
\displaystyle
= \int_0^T \int_{\Omega} \varphi(x_1,x_2,t)\int_{R_1}G(x_1-y_1,x_2)
\left[\int_{R_2}v^{\varepsilon}\left(y_1,\tilde{y_2},t)-u^{\varepsilon}(x_1,x_2,t)\right)d\tilde{y}_2\right]dy_1 dx_2 dx_1 dt.
\end{array}
$$
Let
\begin{equation}\label{Veps}
V^{\varepsilon}(y_1,t)=\int_{R_2}v^{\varepsilon}(y_1,\tilde{y}_2,t)d\tilde{y}_2. 
\end{equation}
Observe that, since $v^{\varepsilon}$ is bounded in $L^\infty (0,T;L^{2}(R))$, then $V^{\varepsilon}$ is also bounded in $L^\infty(0,T;L^2(R_{1}))$ 
so, taking a subsequence if necessary
$$
V^{\varepsilon} \rightharpoonup V^{*} \quad \mbox{weakly} \quad \mbox{in} \quad L^\infty(0,T;L^2(R_1)).
$$
Using \eqref{Veps} we obtain
\begin{align*}
\int_0^T \int_{\Omega}\varphi(x_1,x_2,t)\int_{R_1}G(x_1-y_1,x_2)\left[V^{\varepsilon}(y_1)-|R_{2}|u^{\varepsilon}(x_1,x_2,t)\right]dy_1 dx_2 dx_1dt.
\end{align*}

Therefore, we can take the limit as $\varepsilon \to 0$ and obtain
$$
\begin{array}{l}
\displaystyle 
\lim_{\varepsilon \to 0} \int_0^T\int_{\Omega}\varphi(x_1,x_2,t)\int_{R_1}G_{\varepsilon}(x_1-y_1,x_2-\varepsilon \tilde{y}_2)\left[V^{\varepsilon}(y_1)-|R_{2}|u^{\varepsilon}(x_1,x_2,t)\right]dy_1 dx_2 dx_1 dt \\[10pt] 
\displaystyle 
 = \int_0^T \int_{\Omega}\varphi(x_1,x_2,t)\int_{R_1} \left(G(x_1-y_1,x_2)\right) \lim_{\varepsilon \to 0}\left(V^{\varepsilon}(y_1)\right)dy_1dx_2dx_1 dt  \\[10pt] \displaystyle
 \qquad - \int_0^T |R_2| \int_{\Omega}\varphi(x_1,x_2,t)\int_{R_1} \left(G(x_1-y_1,x_2)\right) \lim_{\varepsilon \to 0}\left(u^{\varepsilon}(x_1,x_2,t)\right)dy_1dx_2dx_1 dt \\[10pt]
\displaystyle = \int_0^T \int_{\Omega}\varphi(x_1,x_2,t)\int_{R_1} G(x_1-y_1,x_2-0)V^{*}(y_1)dy_1dx_2dx_1 dt\\[10pt]
\qquad \displaystyle  - \int_0^T |R_2|\int_{\Omega}\varphi(x_1,x_2,t)\int_{R_1}G(x_1-y_1,x_2-0)u^{*}(x_1,x_2,t)dy_1dx_2dx_1dt. \label{limit-couple-partI}
\end{array}
$$

The same idea can be applied for the second integral in the weak form of the problem using the properties of the kernel $G$ and Fubini's theorem, which leads to 
\begin{align}\label{limit-couple-partII}
\int_0^T \int_{R_1}\varphi(x_1,t)\int_{\Omega} G(x_1-y_1,0-y_2)\left(V^{*}(y_1)-|R_2|u^{*}(x_1,x_2,t)\right)dy_1dx_2dx_1 dt.
\end{align}

Concerning the terms that involve time derivatives, from the $L^\infty-L^2$ convergence we obtain
$$
\lim_{\varepsilon \to 0} - \int_0^T \int_\Omega u^\varepsilon (x,t) \frac{\partial \varphi}{\partial t} (x,t) dxdt
= - \int_0^T \int_\Omega u^* (x,t) \frac{\partial \varphi}{\partial t} (x,t) dxdt
$$
and 
$$
\lim_{\varepsilon \to 0} - \int_0^T \int_R v^\varepsilon (x,t) \frac{\partial \varphi}{\partial t} (x,t) dxdt
= - \int_0^T \int_{R_0} V^* (x,t) \frac{\partial \varphi}{\partial t} (x,t) dxdt
$$

Finally, we will deal with the pure nonlocal integral. By Fubini's theorem and \eqref{Veps} we get
$$
\begin{array}{l}
\displaystyle
\int_0^T \int_{R_1}\int_{R_2}\int_{R_1}\int_{R_2} J_{\varepsilon}(x-y)\left(v^{\varepsilon}(y_1,\tilde{y}_2,t)-v^{\varepsilon}(x_1,\tilde{x}_2,t) \right)\varphi(x_1,t) d\tilde{y}_2 dy_1 d\tilde{x}_2 dx_1 dt \\[10pt] \displaystyle
= \int_0^T \int_{R_1}\varphi(x_1,t) \int_{R_1} J(x_1-y_1,\varepsilon(\tilde{x}_2,\tilde{y}_2))\left[\int_{R_2}\int_{R_2}\left(v^{\varepsilon}(y_1,\tilde{y}_2,t)-v^{\varepsilon}(x_1,\tilde{x}_2,t) \right)d\tilde{y}_2 d\tilde{x}_2\right]  dy_1  dx_1 dt \\[10pt] 
\displaystyle
= \int_0^T \int_{R_1}\varphi(x_1,t) \int_{R_1} J(x_1-y_1,\varepsilon(\tilde{x}_2,\tilde{y}_2))\left(|R_2|V^{\varepsilon}(y_1)-|R_2|V^{\varepsilon}(x_1) \right)  dy_1  dx_1 dt . 
\end{array}
$$
Now, we can take the limit as $\varepsilon \to 0$, it follows that
$$
\begin{array}{l}
\displaystyle
\lim_{\varepsilon \to 0}\int_0^T\int_{R_1}\varphi(x_1,t) \int_{R_1} J(x_1-y_1,\varepsilon(\tilde{x}_2,\tilde{y}_2))\left(|R_2|V^{\varepsilon}(y_1)-|R_2|V^{\varepsilon}(x_1) \right)  dy_1  dx_1 dt \\[10pt] \displaystyle
= \int_0^T |R_2|\int_{R_1}\varphi(x_1,t)\int_{R_1}\lim_{\varepsilon \to 0}\left(J(x_1-y_1,\varepsilon(\tilde{x}_2,\tilde{y}_2))\right)\lim_{\varepsilon \to 0}(V^{\varepsilon}(y_1))dy_1dx_1 dt \\[10pt]
\displaystyle \qquad - \int_0^T |R_2|\int_{R_1}\varphi(x_1,t)\int_{R_1}\lim_{\varepsilon \to 0}\left(J(x_1-y_1,\varepsilon(\tilde{x}_2,\tilde{y}_2))\right)\lim_{\varepsilon \to 0}(V^{\varepsilon}(x_1))dy_1dx_1 dt\\[10pt]
\displaystyle
= \int_0^T |R_2|\int_{R_1}\varphi(x_1,t)\int_{R_1}J(x_1-y_1,0)(V^{*}(y_1) - V^{*}(x_1))dy_1dx_1 dt. 
\end{array}
$$
Hence, since this procedure can be carry over for every $T>0$, the limit equation, defined in the domain $\Omega_0 = \Omega \cup R_{1}$, (see Figure \ref{fig2}) is given by the system \eqref{limit-problem-neumann},
\begin{align*}
\begin{cases}
 \displaystyle   \frac{\partial u^{*}}{\partial t}(x,t)  = \Delta u^{*}(x,t) + \int_{R_{1}}G^{*}(x-y)(V^{*}(y,t)-|R_2|u^{*}(x,t))dy, \quad (x,t) \in \Omega\times (0,\infty),  \\[10pt]
 \displaystyle   \frac{\partial u^{*}}{\partial \eta}(x,t)  =  0, \quad (x,t) \in \partial \Omega \times (0,\infty),  \\[10pt]
 \displaystyle \frac{\partial V^{*}}{\partial t}(x,t)=|R_2|\int_{R_1} J^{*}(x-y)\left(V^{*}(y,t)-V^{*}(x,t) \right)dy 
  \\[10pt]
  \displaystyle \qquad \qquad \qquad \qquad  - \int_{\Omega}G^{*}(x-y)(V^{*}(x,t)- |R_2| u^{*}(y,t))dy, \ (x,t) \in R_1\times (0,\infty), \\[10pt] 
 \displaystyle    u^{*}(x,0)  =u_{0}^{*}(x), \quad x \in \Omega, \\[5pt]
 V^{*}(x,0)  =V^{*}_{0}(x), \quad x \in R_1,
    \end{cases}
    \end{align*}
where $J^{*}(x-y) = J(x_1-y_1,0)$ and $G^{*}(x-y) = G(x_1-y_1,x_2-0)$.

To finish the proof we show existence and uniqueness of a solution of the solution to the limit problem \eqref{limit-problem-neumann}
(notice that up to this point we have convergence along subsequences $\varepsilon_j \to 0$, proving uniqueness 
of the limit we obtain the existence of the full limit as $\varepsilon \to 0$).

Thanks to the limit along subsequences we ensure the existence of a solution $(u^{*},V^{*})$ for the limit problem. To show the uniqueness let us suppose that there exists two solutions $(u^{*}_{1},V^{*}_{1})$ and $(u^{*}_{2},V^{*}_{2})$ of \eqref{limit-problem-neumann}. Define $w^{*} = u_{1}^{*}- u_{2}^{*}$ and $z^{*} = V_{1}^{*}- V_{2}^{*}$. The pair of $(w^{*},z^{*})$ satisfies the following equations
\begin{align}\label{diff-system}
\begin{cases}
 \displaystyle   \frac{\partial w^{*}}{\partial t}(x,t)  = \bigtriangleup w^{*}(x,t) + \int_{R_{1}}G^{*}(x-y)(z^{*}(y,t)-|R_2|w^{*}(x,t))dy, 
 \quad (x,t) \in \Omega \times (0,\infty),  \\[10pt]
 \displaystyle   \frac{\partial w^{*}}{\partial \eta}(x,t)  =  0, \quad  (x,t) \in \partial \Omega \times (0,\infty),  \\[10pt]
 \displaystyle \frac{\partial z^{*}}{\partial t}(x,t)=|R_2|\int_{R_1} J^{*}(x-y)\left(z^{*}(y,t)-z^{*}(x,t) \right)dy 
   - \int_{\Omega}G^{*}(x-y)(z^{*}(x,t)- |R_2| w^{*}(y,t))dy, \\[10pt]
   \qquad\qquad\qquad\qquad\qquad\qquad \qquad\qquad\qquad\qquad (x,t) \in R_1 \times (0,\infty), \\[10pt] 
 \displaystyle    w^{*}(x,0)  = 0, \quad x \in \Omega, \\[5pt]
 z^{*}(x,0)  =0, \quad x \in R_1.
    \end{cases}
    \end{align}
Multiplying the first equation of the problem \eqref{diff-system} by $\frac{w^{*}}{2}$ and integrating over $\Omega$ and, the second equation by $\frac{z^{*}}{2}$ and integrating over $R_{1}$, we get
\begin{align}
|R_2|\int_{\Omega}\frac{\partial w^{*}}{\partial t}w^{*}dx+\int_{R_{1}}\frac{\partial z^{*}}{\partial t}z^{*}dx & = -|R_{2}|\int_{\Omega}|\nabla w^*|^{2}dx
-\frac{|R_2|}{2}\int_{R_1}\int_{R_1}J^{*}(x-y)(z^{*}(y,t)-z^{*}(x,t))^{2}dydx \\
& \qquad - \int_{R_{1}}\int_{\Omega}G^{*}(x-y)(z^{*}(y,t)-|R_2|w^{*}(x,t))^2dydx \\
& = -2 E(w^{*},z^{*}) \leq 0.
\end{align}
Hence, if we let $$f'(t) = |R_2|\int_{\Omega}\frac{\partial w^{*}}{\partial t}w^{*}dx+\int_{R_{1}}\frac{\partial z^{*}}{\partial t}z^{*}dx,$$ 
we have $$f(t) = \frac{|R_2|}{2}\int_{\Omega}(w^{*})^2dx+\int_{R_{1}}(z^{*})^2dx.$$

Now, from Lemma \ref{positive-constant}, we obtain
\begin{align*}
2 E(w^{*},z^{*}) \geq 2 \lambda_{1} \left(\frac{|R_2|}{2}\int_{\Omega}(w^{*})^2dx+\int_{R_{1}}(z^{*})^2dx\right) = 2 \lambda_{1} f(t),
\end{align*}
which implies $-2 E(w^{*},z^{*}) \leq -2 \lambda_{1} f(t)$ and then we get
\begin{align}
f'(t) \leq -2 \lambda_{1} f(t). \label{ineq}
\end{align}
Hence, Gronwall's inequality gives that
$$
f(t) \leq e^{-2 \lambda_{1} t}f(0),
$$
where $f(0) = \frac{|R_2|}{2}\int_{\Omega}(w^{*})^2(x,0)dx+\int_{R_{1}}(z^{*})^2(x,0)dx$. Since $f(t)\geq 0$ and $f(0)=0$ we have that
$$
0 \leq f(t) \leq 0,
$$ 
that is $$f(t)\equiv 0$$ and hence $$w^{*}=0 \qquad \mbox{ and } \qquad z^{*}=0,$$ which means $u_{1}^{*}=u^{*}_2$ and $V_{1}^{*}=V_{2}^{*}$. This guarantee the uniqueness of the solution for the problem \eqref{limit-problem-neumann} as we wanted to show.
\end{proof}

Now, we include several remarks. 

\begin{remark}
{\rm From our previous arguments, we also conclude that the limit problem \eqref{limit-problem-neumann} is well-posed in $L^{2}(\Omega_{0})$
(we have existence, uniqueness and continuous dependence with respect to the initial data of the solutions).} 
\end{remark}

\begin{remark} {\rm
We only prove weak convergence of the solution of the problem \eqref{eps-case1} to the solution of the problem \eqref{limit-problem-neumann} 
(we do not prove strong convergence in the $L^2$-norm). Moreover, we only guarantee the uniqueness of $V^{*}$ and this is not enough to ensure the uniqueness of $v^{*}$.}
\end{remark}

\begin{remark} {\rm
Observe that, instead of the usual metric in $L^2 (\Omega \cup R)$ we choose to work with the metric \eqref{different-metric}. This choose was made to obtain a nontrivial limit. In fact, using this metric we can observe the coupling of the local part of the problem in the domain $\Omega$
with the nonlocal part in the lower dimensional domain $R_1$. 

Now, if we consider the usual metric in $L^2$ and the energy functional
$$
E(u,v) = \frac{1}{2}\int_{\Omega} |\nabla  u |^2 dx +\frac{1}{4 \varepsilon^{N_2}}\int_{R_{\varepsilon}}\int_{R_{\varepsilon}}J(x-y)\left(v(y)-v(x)\right)^2 dydx + \frac{1}{2}\int_{R_{\varepsilon}}\int_{\Omega}G(x-y)\left(v(x)-u(y)\right)^2 dy dx,
$$
the associated evolution problem (after the change of variables) is given by
\begin{align}\label{new-complete-fixed}
\begin{cases}
 \displaystyle   \frac{\partial u^{\varepsilon}}{\partial t}(x,t)  = \bigtriangleup u^{\varepsilon}(x,t) + \varepsilon\int_{R}G_{\varepsilon}(x-y)(v^{\varepsilon}(\tilde{y},t)-u(x,t))d\tilde{y}, \quad x \in \Omega, \quad t >0,  \\[10pt]
 \displaystyle   \frac{\partial u^{\varepsilon}}{\partial \eta}(x,t)  =  0, \quad x \in \partial \Omega, \quad t>0,  \\[10pt]
 \displaystyle \frac{\partial v^{\varepsilon}}{\partial t}(\tilde{x},t)=\int_{R} J_{\varepsilon}(x-y)\left(v^{\varepsilon}(\tilde{y},t)-v^{\varepsilon}(\tilde{x},t) \right)d\tilde{y} 
   - \int_{\Omega}G_{\varepsilon}(x-y)(v^{\varepsilon}(\tilde{x},t)-  u^{\varepsilon}(y,t))dy, \quad x \in R, \quad t>0,  \\[10pt]
 \displaystyle    u(x,0)  =u_{0}(x), \quad x \in \Omega, \\[5pt]
 v(\tilde{x},0)  =v_{0}(\tilde{x}), \quad \tilde{x} \in R,
    \end{cases}
    \end{align}
where $J_{\varepsilon}(x-y)=J(x_1-y_1,\varepsilon(\tilde{x}_2-\tilde{y}_2))$, $G_{\varepsilon}(x-y)=G(x_1-y_1,\varepsilon \tilde{x}_2-\tilde{y}_2)$ and $v^{\varepsilon}(x_1,\tilde{x}_2,t)=v(x_1,\varepsilon \tilde{x}_2,t)$. Observe that taking the limit as $\varepsilon \to 0$ the nonlocal term that appears in the  equation for $u^{\varepsilon}$ goes to zero and hence we will lose the coupling term in the limit (the equation for $u^{*}$ will be independent of $V^{*}$). Also in this case, the limit problem will be well-posed, in the sense that we can ensure existence and uniqueness of the solution, but it is
less interesting.}
\end{remark}

As we expected, the limit problem \eqref{limit-problem-neumann} preserves the total mass of the solution.
This follows from the limit procedure and the fact that the problem \eqref{neumann-fixed-domain} preserves the total mass for every $\varepsilon >0$. 
We include below a direct proof of this fact for completeness. 

\begin{theorem}\label{mass-theo}
The solution $(u^{*},V^{*})$ of the problem 
\eqref{limit-problem-neumann}, with initial data $u_{0}^{*} \in H^{1}(\Omega)$ and $V^{*}_{0} \in L^{2}(R_1)$ satisfies
\begin{equation}\label{mass}
\int_{\Omega}u^{*}(x,t)dx+\int_{R_1}V^{*}(x,t)dx = \int_{\Omega}u^{*}_0 dx+\int_{R_1}V^{*}_0 dx, \quad \forall t \geq 0.
\end{equation}
\end{theorem}

\begin{proof}
Differentiating \eqref{mass} with respect to $t$ we obtain
\begin{align*}
\int_{\Omega}\frac{\partial u^{*}}{\partial t} dx+\int_{R_1}\frac{\partial V^{*}}{\partial t}dx &= \int_{\Omega} \bigtriangleup u^{*}(x,t)dx + \int_{\Omega}\int_{R_1}G^{*}(x-y)(V^{*}(y,t)-|R_2|u^{*}(x,t))dydx \\
& \qquad +\int_{R_1}|R_2|\int_{R_1}J^{*}(x-y)(V^{*}(y,t)-V^{*}(x,t))dydx \\
& \qquad - \int_{R_1}\int_{\Omega}G^{*}(x-y)(V^{*}(x,t)-|R_2|u^{*}(y,t))dydx \\
& = 0.
\end{align*}
Indeed, after a change of variables, due to the symmetry of $G$ and Fubini's theorem, the second and the fourth integral cancel each other. 
Also, by the symmetry of $J$ and Fubini's theorem, the second integral is zero. Finally, the first integral is zero since we have a Neumann type boundary condition for the local part.

This ends the proof.
\end{proof}

Finally, we include the study of the asymptotic behavior of the solutions for the limit problem \eqref{limit-problem-neumann}.

Notice that from the fact that the constants in \eqref{decay.exp} do not depend on $\varepsilon$ we obtain 
that the solutions for the limit problem \eqref{limit-problem-neumann}
converge exponentially to the mean value of the initial condition. We have that
\begin{equation} \label{decay.exp.789}
 \Big\| (u^{*},V^{*})(\cdot,t)-\fint (u_0,v_0) \Big\|_{L^2(\Omega_0)} \leq C_{1} e ^{-C_2 t}.
\end{equation}

However, we can obtain a better control of the constant $C_2$ and obtain an exponential decay
in terms of the first nontrivial eigenvalue associated to the limit problem.
To this end, we use the $L^2$-norm
\begin{equation}\label{lim-prob-norm}
\| (u^*,V^{*})\|_{L^2(\Omega_0)} = |R_2|\int_{\Omega}|u^{*}|^2dx+\int_{R_1}|V^{*}|^2 dx.
\end{equation}
We can define the energy functional associated to the limit problem \eqref{limit-problem-neumann} by
\begin{equation}\label{neumann-limit-energy}
\begin{array}{l}
\displaystyle
E(u^{*},V^{*}) = \frac{|R_2|}{2}\int_{\Omega}|\nabla u^{*}|^2 dx + \frac{|R_2|}{4}\int_{R_1}\int_{R_1}J^{*}(x-y)(V^{*}(y)-V^{*}(x))^2 dydx \\[10pt]
\qquad\qquad\qquad\qquad \displaystyle  +\frac{1}{2}\int_{R_1}\int_{\Omega}G^{*}(x-y)(V^{*}(x)-|R_2|u^{*}(y))^2dydx.
\end{array}
\end{equation}
Indeed, the gradient flow associated with \eqref{neumann-limit-energy}, is given by 
\begin{align*}
 \displaystyle 
 \partial_{\varphi}E(u^{*},V^{*})  & \displaystyle =\lim_{h \to 0}\frac{E(u^{*}+h\varphi,V^{*}+h\varphi)-E(u^{*},V^{*})}{h} \\
 \displaystyle  & = |R_2| \int_{\Omega} \nabla u^{*} \nabla \varphi dx + \frac{|R_2|}{2}\int_{R_1}\int_{R_1}J^{*}(x-y)(V^{*}(y)-V^{*}(x))(\varphi(y)-\varphi(x))dydx  \\
 \displaystyle  & \qquad + |R_2|\int_{R_1}\int_{\Omega}G^{*}(x-y)(V^{*}(x)-|R_2|u^{*}(y))(\varphi(x)dx - |R_2| \varphi(y))dydx.
\end{align*}
Hence, using that
$$
|R_2|\int_{\Omega}\frac{\partial u^{*}}{\partial t}\varphi(x) dx +\int_{R_1}\frac{\partial V^{*}}{\partial t}\varphi(x)dx = -\partial_{\varphi}E[(u^{*},V^{*})(t)],
$$
we obtain the limit problem \eqref{limit-problem-neumann}.

With this energy at hand we can obtain the first nontrivial eigenvalue for our limit problem. Let us take $\lambda_{1}$ as 
\begin{equation}\label{neumann-lim-eig}
0<\lambda_{1} = \inf_{u^{*},V^{*} \in W_0}\frac{E(u^{*},V^{*})}{ \displaystyle |R_2|\int_{\Omega}(u^{*})^2 + \int_{R_1}(V^{*})^2},
\end{equation}
where $E(u^{*},V^{*})$ is given by \eqref{neumann-limit-energy} and
$$
W_0 = \Big\{u^{*} \in H^1(\Omega), V^{*} \in L^2(R_1): |R_2|\int_{\Omega}u^{*}+\int_{R_1}V^{*}=0 \Big\}.
$$

\begin{lemma}\label{c-limit}
Let $\lambda_{1}$ given by \eqref{neumann-lim-eig}, then $\lambda_{1} >0$ and therefore,
$$
E(u^{*},V^{*}) \geq \lambda_{1} \left(|R_2|\int_{\Omega} (u^{*})^2+\int_{R_1}(V^{*})^2\right),
$$
for every $u^{*}, V^{*}$ solution of \eqref{limit-problem-neumann}, such that $|R_2|\int_{\Omega}u^{*}+\int_{R_1}V^{*}=0.$
\end{lemma}

\begin{proof} The proof is similar to the one of Lemma \ref{positive-constant} but we include the details for completeness. 
Let us suppose that $\lambda_{1} = 0$. This implies that there exists a subsequence $\{u_n^{*}\} \in H^{1}(\Omega)$ and 
$\{v_n^{*}\} \in L^2(R_1)$ such that
$$
|R_2|\int_{\Omega}u^{*}_n+\int_{R}V^{*}_n=0,
$$
$$
|R_2|\int_{\Omega}(u^{*}_n)^2+\int_{R}(V^{*}_n)^2=1,
$$
and
$$
\frac{|R_2|}{2}\int_{\Omega}|\nabla u^{*}_n|^2 dx+\frac{|R_2|}{4}\int_{R_1}\int_{R_1} J^{*}(x-y)(V^{*}_n(y)-V^{*}_n(x))^2dydx+\frac{1}{2}\int_{R_1}\int_{\Omega}G^{*}(x-y)(V^{*}_n(x)-|R_2|u^{*}_n(y))^2dydx \leq \frac{1}{n}.
$$
Taking the limit as $n\to \infty$ we obtain
$$
\lim_{n \to \infty}\left(\frac{|R_2|}{2}\int_{\Omega}|\nabla u^{*}_n|^2 dx\right)=0,
$$
$$
\lim_{n \to \infty} \frac{|R_2|}{4}\int_{R_1}\int_{R_1} J^{*}(x-y)(V^{*}_n(y)-V^{*}_n(x))^2dydx =0,
$$
and
$$
\lim_{n \to \infty}\left(\frac{1}{2}\int_{R_1}\int_{\Omega}G^{*}(x-y)(V^{*}_n(x)-|R_2|u^{*}_n(y))^2dydx\right)=0.
$$

Recalling that we have
$$
J^{*}(x-y) = J(x_1-y_1,0), \qquad \mbox{and} \qquad 
G^{*}(x-y) = G(x_1-y_1,0-y_2),
$$
it follows that $|R_2|\int_{\Omega}(u^{*}_n)^2 dx \leq 1$, that is, $\{u^{*}_n\}$ is bounded in $L^2(\Omega)$. 
Moreover, $\{u^{*}_n\}$ is also bounded in $H^1(\Omega)$. Then, we can extract a subsequence $\{u^{*}_{n_j}\} \in H^1(\Omega)$ 
which weakly converges to a limit $\hat{u} \in H^1(\Omega)$. From the weak convergence in $H^1(\Omega)$ we obtain strong convergence in $L^2(\Omega)$.
Then, we have that
$$
\frac{1}{2} |R_2|\int_{\Omega}|\nabla \hat{u}|^2 dx \leq \liminf_{n}\frac{1}{2} |R_2| \int_{\Omega}|\nabla u^{*}_n|^2 dx = 0.
$$
Hence, the limit $\hat{u}$ is constant in $\Omega$.

Also, it follows that $\{V^{*}_n\}$ is bounded in $L^{2}(R_1)$. Since 
$$\int_{R}|V^{*}_n|dx \leq C \left(\int_{\Omega}(V^{*}_n)^2 dx\right)^\frac{1}{2} \leq C,$$
we let
$k_{n} = \int_{R_1}V^{*}_n$, and obtain that $|k_n|\leq C$. Then, we can take a subsequence $\{V^{*}_{n_j}\}$ which converges in $L^2(R_1)$, 
to some limit $\hat{V}$ as $n_j \to \infty$. Consider $z_{n_{j}} = V^{*}_{n_j}-k_{n_j}$, this function is such that $\int_{R_1}k_{n_j}=0$.
By Lemma 3.1, in \cite{CERW}, there exists a constant $c_1>0$ such that 
$$
\int_{R_1}\int_{R_1}J(x_1-y_1,0)(z_{n_{j}}(y)-z_{n_{j}}(x))^2dy dx \geq c_1 \int_{R_1}(z_{n_{j}}(x))^2 dx.
$$
From this inequality we have
\begin{align*}
\lim_{n \to \infty}\left(\frac{|R_2|}{4}\int_{R_1}\int_{R_1} J^{*}(x-y)(z_{n_{j}}(y)-z_{n_{j}}(x))^2dydx\right)=\lim_{n \to \infty}\left(\frac{|R_2|}{4}\int_{R_1}\int_{R_1} J^{*}(x-y)(V^{*}_{n_j}(y)-V^{*}_{n_j}(x))^2dydx\right) \to 0,
\end{align*}
which yields
$$
0 \geq c_1 \lim_{n \to \infty} \int_{R_1}\int_{R_1}(z_{n_{j}}(x))^2 dx.
$$
We conclude that $z_{n_{j}} \to 0$ strongly in $L^2(R_1)$, which leads to $V^{*}_{n_j} \to \hat{V}$ strongly $ \in L^2(R_1)$.
Finally, as $u^{*}_n \to \hat{u}$ in $L^2(\Omega)$ and $V^{*}_{n} \to \hat{V}$ in $L^2(R_1)$, 
we can take the limit as $n \to \infty$ and obtain
\begin{align*}
\lim_{n \to \infty}\left(\frac{1}{2}\int_{R_1}\int_{\Omega}G^{*}(x-y)(V^{*}_{n}(x)-u^{*}_{n}(y))^2dydx\right)=\frac{1}{2}\int_{R_1}\int_{\Omega}G^{*}(x-y)(\hat{V}-|R_2|\hat{u})^2dydx \to 0.
\end{align*}
Then, we have that $\hat{V}-|R_2|\hat{u}=0$, that is $\hat{V}=|R_2|\hat{u}$. Hence, it follows that $\hat{V}=\hat{u}=0$, 
but this is a contradiction with the fact that 
$$
|R_2|\int_{\Omega}(u^{*}_n)^2+\int_{R}(V^{*}_n)^2=1
$$
since we have strong convergence in $L^2$. 
The proof is complete.
\end{proof}

Thanks to Lemma \ref{c-limit} we can show that solutions to the limit problem 
converge exponentially fast to the mean value of their initial condition.

\begin{theorem}
Given $u^{*}_{0} \in H^1(\Omega)$ and $V^*_0 \in L^2 (R_1)$, the solution to \eqref{limit-problem-neumann}, with initial data $u^{*}_{0},V^*_0$, 
converges to its mean value as $t \to \infty$, with an exponential rate $\lambda_1$ (given by \eqref{neumann-lim-eig}),
$$
\Big\| (u^{*},V^*)(\cdot,t)-\fint (u_{0}^{*},V^*_0) \Big\|_{L^2({\Omega}_0)} \leq C\left(\| (u^{*}_0,V^*_0) \|_{L^2(\Omega_0)}\right) 
 e ^{- \lambda_{1} t}.
$$
\end{theorem}

\begin{proof}
We know that $V^{*} = |R_2|u^{*} = k$, with $k$ constant, is also a solution of the problem \eqref{limit-problem-neumann}. Hence, the pair
$$(h(x,t) = |R_2|u^{*}(x,t)-k, z(x,t) = V^{*}(x,t)-k)$$ 
is also a solution of \eqref{limit-problem-neumann}. If we choose $$k = |R_2|\int_{\Omega}u_0^{*}+\int_{R_1}V^{*}_0$$ then, 
using that the mass is preserved in time, we get that $h$ and $z$ satisfy
$$
\int_{\Omega} h(x,t) dx+\int_{R_1}z(x,t)dx = 0.
$$
Let $$f(t) = \frac{|R_2|}{2}\int_{\Omega}h(x,t)^2 dx+\frac{1}{2}\int_{R_1}z(x,t)^2dx.$$ 
Differentiating $f$ with respect to $t$ we obtain
\begin{align*}
f'(t) &= |R_2|\int_{\Omega}\frac{\partial h}{\partial t}(x,t)h(x,t)dx+\int_{R_1}\frac{\partial z}{\partial t}(x,t)z(x,t) dx \\
& = |R_2|\int_{\partial \Omega} \frac{\partial h}{\partial \eta}(x,t)h(x,t) dx - |R_2|\int_{\Omega} | \nabla h(x,t)|^2 dx 
-\frac{|R_2|}{2}\int_{R_1}\int_{R_1}J^{*}(x-y)(z(y,t)-z(x,t))^2 dy dx \\
& \quad -\int_{R_1}\int_{\Omega}G^{*}(x-y)(z(x,t)-|R_2|h(y,t))z(x,t)dydx + \int_{R_1}\int_{\Omega}G^{*}(x-y)(z(x,t)-|R_2|h(y,t))|R_2|h(x,t)dydx \\
& = |R_2|\int_{\Omega} |\nabla h(x,t)|^2 dx -\frac{|R_2|}{2}\int_{R_1}\int_{R_1}J^{*}(x-y)(z(y,t)-z(x,t))^2 dy dx \\
& \quad - \int_{R_1}\int_{\Omega}G^{*}(x-y)(z(x,t)-|R_2|h(y,t))^2dydx \\
& = - 2E(h,z).
\end{align*}
From Lemma \ref{c-limit} we get
$$
E(h,z) \geq \lambda_{1} \left(|R_2|\int_{\Omega} h^2+\int_{R_1}z^2\right).
$$
Hence, we obtain 
$$f'(t) \leq -2 \lambda_{1} f(t)$$ so, by Gronwall's lemma we have that $$f(t) \leq e^{-2 \lambda_{1} t}f(0),$$ 
with $f(0) = \frac{1}{2}\left(|R_2|\int_{\Omega}h_0^2+\int_{R_1}z^{2}_0\right)$.
From this it follows that
$$
|R_2|\int_{\Omega}||R_2|u^{*}(x,t)-k|^2 dx+\int_{R_1}|V^{*}(x,t)-k|^{2}dx \leq C\left(\| (u^*_0,V_0^*) \|_{L^2(\Omega_0)}\right) e^{-2 \lambda_{1} t} \longrightarrow 0,
$$
as $t \to \infty$. In particular, it means that $|R_2|u^{*}\rightarrow k$ in $L^2(\Omega)$ and $V^{*}\longrightarrow k$ in $L^2(R_1)$, with $k$ given by the mean value of the initial condition.
\end{proof}

\section{Coupling on the boundary. Proof of Theorem \ref{teo2}} \label{Sect-3}

Let us first note that the existence and uniqueness of the solutions $\{u^{\varepsilon},v^{\varepsilon}\}$, of the problem \eqref{problemrescaled-energy2}, for each $\varepsilon >0$, was obtained in \cite{santos2020}. The arguments used to prove the conservation of mass and comparison principle also apply for the problem \eqref{problemrescaled-energy2} following the ideas presented in \cite{santos2020}.

Notice that we have an energy functional for the problem \eqref{problemrescaled-energy2} given by
\eqref{energy2},
$$
    E^b_\varepsilon (u,v)  \displaystyle:=\frac{1}{2}\int_{\Omega} |\nabla  u |^2 dx 
    + \frac{1}{4\varepsilon ^{2 N_{2}}}\int_{R_\varepsilon}\int_{R_\varepsilon}J(x-y)\left(v(y)-v(x)\right)^2 dydx  + \frac{1}{2\varepsilon^{N_2}}\int_{R_{\varepsilon}}\int_{\Gamma}G(x-y)dy\left(v(x)-u(y)\right)^2 d\sigma(y) dx.
$$
If we change variables as before we get
\begin{equation}\label{neumann-energy.99}
    E^b (u,v)  \displaystyle:=\frac{1}{2}\int_{\Omega} |\nabla  u |^2 dx 
    + \frac{1}{4}\int_{R}\int_{R}J_\varepsilon (x-y) |v(y)-v(x)|^2 dydx  + \frac{1}{2}\int_{R}\int_{\Gamma}G_\varepsilon (x-y)|v(x)-u(y)|^2 dy dx,
\end{equation}
with, as before,
$$
J_{\varepsilon}(x-y) = J(x_1-y_1,\varepsilon(\tilde{x}_2-\tilde{y}_2)),
\qquad
G_{\varepsilon}(x-y) = G(x_1-y_1,\varepsilon\tilde{x}_2-\tilde{y}_2).
$$

Now, we just observe that Lemma \ref{positive-constant} also works here.
One can define what is the analogous to the first non-zero eigenvalue for the problem \eqref{problemrescaled-energy2} as follows:
\begin{equation}\label{eps-eigen-energy2}
\alpha_{1}^{\varepsilon}=\inf_{u^{\varepsilon},v^{\varepsilon} \in \mathcal{A}}
\frac{E^b (u^{\varepsilon},v^{\varepsilon})}{
\displaystyle \int_{\Omega}(u^{\varepsilon})^2dx+\int_{R}(v^{\varepsilon})^{2}d\tilde{x}},
\end{equation}
with $$\mathcal{A} = \Big\{u^{\varepsilon} \in H^{1}(\Omega), v^{\varepsilon} \in L^{2}(R): \int_{\Omega}u^{\varepsilon}dx
+\int_{R}v^{\varepsilon}dx=0\Big\}.$$ 
For the positivity of $\alpha_{1}^{\varepsilon}$ we refer to \cite{andreu2010nonlocal}. A uniform lower bound independent of 
$\varepsilon$ can be proved as in Lemma \ref{positive-constant}.
The large time behavior for the solutions of \eqref{problemrescaled-energy2} can be obtained following the ideas developed in \cite{santos2020}. 
As we find in \cite{santos2020}, the solutions of \eqref{problemrescaled-energy2} converge exponentially 
to the mean value of the initial data as $t$ goes to $\infty$, for each $\varepsilon$.

Now, we prove Theorem \ref{teo2} taking the limit as $\varepsilon$ goes to zero in the weak form of \eqref{problemrescaled-energy2}.

\begin{proof}[Proof of the Theorem \ref{teo2}] We proceed as in the proof of Theorem \ref{teo1}. First, we obtain 
convergence along subsequences. 
From Lemma \ref{positive-constant}, since
$$
\int_{R}(v^{\varepsilon}(x))^2dx
$$
is bounded in $L^2(R)$ we can take a subsequence, also denoted by $\{v^{\varepsilon}\}$, such that 
\begin{subequations}
\begin{align}\label{eq:v.77}
v^{\varepsilon} \rightharpoonup v^{*} \quad \mbox{weakly} \quad \mbox{in} \quad L^\infty (0,T;L^2(R)) \quad \mbox{as} \quad \varepsilon \to 0.
\end{align}
\end{subequations} 
On the other hand, we have that
%\begin{subequations}
\begin{align*}
\int_{\Omega}(u^{\varepsilon}(x))^2dx \qquad \mbox{and} \qquad \int_{\Omega} |\nabla  u^{\varepsilon}(x)|^2dx, 
\end{align*}
%\end{subequations}
are also bounded, and hence $u^{\varepsilon}$ is bounded in $H^1(\Omega)$. Hence, along a subsequence if necessary,
\begin{subequations}
\begin{align}
u^{\varepsilon} \rightharpoonup u^{*} \quad \mbox{weakly} \quad \mbox{in} \quad L^\infty(0,T;H^1(\Omega)) \quad \mbox{as} \quad \varepsilon \to 0, \label{eq:u.77} \\[10pt]
u^{\varepsilon} \to u^{*} \quad \mbox{strongly} \quad \mbox{in} \quad L^\infty(0,T;L^2(\Omega)) \quad \mbox{as} \quad \varepsilon \to 0.
\end{align}
\end{subequations}

Let us consider the weak form of the problem \eqref{problemrescaled-energy2} using for the equation for the variable $v^{\varepsilon}$
(the second equation of \eqref{problemrescaled-energy2}) a test function that depends only on the first variable, that is $\varphi = \varphi(x_1)$. 
We have,
\begin{subequations}
\begin{align}
& \int_{\Omega}\int_{\Omega} u^{\varepsilon}(x_1,x_2,T) \frac{\partial \varphi }{\partial t}(x_1,x_2,T) dx_2dx_1 
- \int_0^T \int_{\Omega}\int_{\Omega} u^{\varepsilon}(x_1,x_2,t) \frac{\partial \varphi }{\partial t}(x_1,x_2,t) dx_2dx_1 \\ 
& =
\int_{\Omega}\int_{\Omega} u_0^{\varepsilon}(x_1,x_2) \varphi (x_1,x_2,0) dx_2dx_1  -
 \int_0^T \int_{\Omega} \int_{\Omega} \nabla u^{\varepsilon} \nabla \varphi (x_1,x_2,t) dx_2dx_1 dt  \\
& \qquad + \int_0^T \int_{\Gamma} \int_{R_1}\int_{R_2}G_{\varepsilon}(x-y)(v^{\varepsilon}(y_1,\tilde{y_2},t)-u^{\varepsilon}(x_2,t))\varphi(x_1,x_2,t) d\tilde{y_2}dy_1d \sigma(x_2) dt,
\end{align}
\end{subequations}

\begin{subequations}
\begin{align}
&   \int_{R}\int_{R} v^{\varepsilon}(x_1,x_2,T) \frac{\partial \varphi }{\partial t}(x_1,x_2,T) dx_2dx_1 
- \int_0^T \int_{R}\int_{R} v^{\varepsilon}(x_1,x_2,t) \frac{\partial \varphi }{\partial t}(x_1,x_2,t) dx_2dx_1 \\
& = \int_{R}\int_{R} v_0^{\varepsilon}(x_1,x_2) \varphi(x_1,x_2,0) dx_2dx_1 \\
& \qquad + \int_0^T \int_{R_1}\int_{R_2}\int_{R_1}\int_{R_2} J_{\varepsilon}(x-y)\left(v^{\varepsilon}(y_1,\tilde{y}_2,t)-v^{\varepsilon}(x_1,\tilde{x}_2,t) \right)\varphi(x_1,t) d\tilde{y}_2 dy_1 d\tilde{x}_2 dx_1dt \\ 
& \qquad - \int_0^T \int_{R_1}\int_{R_2} \int_{\Gamma} G_{\varepsilon}(x-y)(v^{\varepsilon}(x_1,\tilde{x}_2,t)-  u^{\varepsilon}(y_2,t))\varphi(x_1,t) 
d\sigma (y_2) d\tilde{x}_2d\tilde{x}_1 dt.   
\end{align}
\end{subequations}

Now we can take the limit for $\varepsilon \to 0$ in each integral on the right side of the previous equations as we did in Theorem \ref{teo1}.
The only difference appears when we analyze the term  
$$
- \int_0^T \int_{R_1}\int_{R_2} \int_{\Gamma} G_{\varepsilon}(x-y)(v^{\varepsilon}(x_1,\tilde{x}_2,t)-  u^{\varepsilon}(y_2,t))\varphi(x_1,t) d\sigma(y_2) d\tilde{x}_2d\tilde{x}_1 dt.
$$
In this case, we need the fact that we have a well defined and compact trace operator $Tr: H^1(\Omega) \mapsto L^2 (\Gamma)$, see \cite{evans}, therefore
from the weak convergence 
$$
u^{\varepsilon} \rightharpoonup u^{*} \quad \mbox{weakly} \quad \mbox{in} \quad L^\infty(0,T;H^1(\Omega)) \quad \mbox{as} \quad \varepsilon \to 0,
$$
we obtain that, along a subsequence,
$$
u^{\varepsilon} \to u^{*} \quad \mbox{strongly} \quad \mbox{in} \quad L^\infty(0,T;L^2(\Gamma)) \quad \mbox{as} \quad \varepsilon \to 0.
$$

As before, since $v^{\varepsilon}$ is bounded in $L^\infty (0,T;L^{2}(R))$, then $V^{\varepsilon}$ is also bounded in $L^\infty(0,T;L^2(R_{1}))$ 
so, taking a subsequence if necessary
$$
V^{\varepsilon} \rightharpoonup V^{*} \quad \mbox{weakly} \quad \mbox{in} \quad L^\infty(0,T;L^2(R_1)).
$$
Using \eqref{Veps} we obtain
\begin{align*}
- \int_0^T \int_{R_1} \int_{\Gamma} G_{\varepsilon}(x-y)(V^{\varepsilon}(x_1,\tilde{x}_2,t)- |R_2| 
u^{\varepsilon}(y_2,t))\varphi(x_1,t) d\sigma (y_2) d\tilde{x}_1 dt.
\end{align*}
Now we can take the limit as $\varepsilon \to 0$ to obtain
$$
\begin{array}{l}
\displaystyle 
\lim_{\varepsilon \to 0} - \int_0^T \int_{R_1} \int_{\Gamma} G_{\varepsilon}(x-y)(V^{\varepsilon}(x_1,\tilde{x}_2,t)- |R_2| 
u^{\varepsilon}(y_2,t))\varphi(x_1,t) d\sigma (y_2) d\tilde{x}_1 dt \\[10pt] 
\displaystyle 
 = - \int_0^T \int_{R_1} \int_{\Gamma} G_{\varepsilon}(x-y)(V^{*}(x_1,\tilde{x}_2,t)- |R_2| 
u^{*}(y_2,t))\varphi(x_1,t) d\sigma (y_2) d\tilde{x}_1 dt. \label{limit-couple-partI.88}
\end{array}
$$

The rest of the terms can be handled as in Theorem \ref{teo1} to obtain
the weak form of the equations of the limit problem \eqref{limit-prob-neumann/robin}. 

Uniqueness of solutions to the limit problem can be obtained as in Theorem \ref{teo1} using the
energy
\begin{equation} \label{energy*}
\begin{array}{l}
\displaystyle E(u^{*},V^{*}) = \frac{|R_2|}{2}\int_{\Omega} |\nabla u^{*}|^{2}dx + \frac{|R_2|}{4}\int_{R_1} \int_{R_1} J^{*}(x-y)\left(V^{*}(y)+V^{*}(x)\right)^{2} dy dx \\[10pt]
\qquad\qquad\qquad\qquad\qquad \displaystyle +\frac{1}{2}\int_{R_1} \int_{\Gamma}G^{*}(x-y)(V^{*}(x)- |R_2|u^{*}(y))^{2} d\sigma (y) dx
\end{array}
\end{equation}
This completes the proof.
\end{proof}

Now, we gather some properties of the limit problem, \eqref{limit-prob-neumann/robin}.

The existence and uniqueness of the limit problem \eqref{limit-prob-neumann/robin} can be obtained using a 
fixed point argument as it was done in \cite{santos2020}. 

The mass conservation in time follows as in the previous section (see also  \cite{santos2020}).

To deal with the large time behavior we can be proceed as we did before, since the solutions of the problem \eqref{limit-prob-neumann/robin} also converges to the mean value of its initial condition as $t$ goes to zero. Notice that from \eqref{limit-prob-neumann/robin} we can define the associated eigenvalue problem. Let us consider $\alpha_{1}^{*}$ given by
\begin{equation}\label{eps-eigen-energy2.99}
\alpha_{1}^{*}=\inf_{u^{*},V^{*} \in \mathcal{A}_{0}}\frac{E(u^{*},V^{*})}{ \displaystyle |R_2|\int_{\Omega}(u^{*})^2dx+\int_{R_1}(V^{*})^{2}dx},
\end{equation}
where $E$ is given by \eqref{energy*} and 
$$
\mathcal{A}_{0} = \Big\{u^{*} \in H^{1}(\Omega), V^{*} \in L^{2}(R_1): |R_2|\int_{\Omega}u^{*}dx+\int_{R_1}V^{*}dx=0 \Big\}.
$$
One can show that $\alpha_1^*$ is strictly positive and then, the computations of the previous
rection can be adapted to prove that the solution of the limit problem \eqref{limit-prob-neumann/robin} converges exponentially fast for the mean value of the initial datum as $t$ goes to $\infty$.

\section{Numerical experiments} \label{sect-numer}
In this section we propose a discrete numerical scheme for the two models, \eqref{limit-problem-neumann} and \eqref{limit-prob-neumann/robin} described in this paper. To obtain a fully discretization of the equations in space and time we will use classical methods, centered finite differences for the interior points of the local part, forward and backward differences for the boundary points; while for the nonlocal region and the coupling terms we just approximate the involved integrals by Riemann sums. 
We use an explicit Euler discretization for the time variable.

As we mentioned in the Introduction, the continuous problems \eqref{limit-problem-neumann} and \eqref{limit-prob-neumann/robin} have some properties: well-posedness, comparison principle, conservation of mass and convergence to the mean value of the initial datum. 
In this section we will perform numerical simulations that illustrate these properties.
 
We will assume that $\Omega$ is a bidimensional rectangle $\Omega = \Omega_1 \times \Omega_2 = [a,b] \times [c,d]$ and we take the mesh parameter as $h_1 = \frac{b-a}{M-1}=\frac{d-c}{N-1}$. Let $h_1$, be same in the two directions. For the nonlocal part, 
the domain $R_1$ will be the segment $R_1 = [b,f]$, with $h_2 = \frac{f-b}{M-1}$.The time step $\triangle_t$ is give by the difference between the final time, $t_f$, with the initial time, $t_0$.

We approximate the continuous solution $u(x,y,t)$, for $(x,y,t) \in \Omega \times \mathbb{R}$ and $V(x,t)$, for $(x,t) \in R_1 \times \mathbb{R}$, by discrete values $u_{i,j}^{l} \approx u(x_i, y_j,t_l)$ and $V_{k}^{l} \approx V(z_k,t_l)$, respectively, with $i,k = 1, \cdots, M$, $j = 1, \cdots, N$.  For simplicity, let us consider a uniform mesh for the local and nonlocal part. The local domain $\Omega$ was discretized by the mesh $(x_{i},y_j)$, with $i = 1, \cdots, M$, $j = 1, \cdots, N$ while, the nonlocal domain, $R_1$ is discretized by the points $z_k$, $k=1, \cdots, M$. 

We will consider $h = h_1 = h_2$ (for simplicity). With this in mind, we call $x_1 = a$, $x_i = x_{i-1}+h$ and $x_{M} = x_{M-1}+h$, $y_1 = c$, $x_i = y_{j-1}+h$ and $y_{N} = y_{N-1}+h$, $z_1 = b$, $z_k = z_{k-1}+h$ and $z_{M} = z_{M-1}+h$.

Then, the numerical approximation of the problem \eqref{limit-problem-neumann}, 
is given by the following system of equations: for the local part we have,
\begin{align}\label{discret1}
\begin{cases}
 \displaystyle   u_{i,j}^{l+1} = u_{i,j}^{l}+\frac{\triangle_t}{h^2} \left(u_{i+1,j}^{l}+u_{i-1,j}^{l}+u_{i,j+1}^{l}+u_{i,j-1}^{l}-4u_{i,j}^{l}\right) \\[7pt]
 \qquad \qquad \qquad  \displaystyle   + \triangle_{t}h\sum_{k=1}^{M}G(x_i-z_{k},y_j)\left(V_{k}^{l}-|R_2|u_{i,j}^{l}\right), \quad i=2, \cdots, M-1, \quad j=2, \cdots, N-1\\[7pt]
 \displaystyle  u_{1,1}^{l+1} = u_{1,1}^{l}+\frac{\triangle_t}{h^2} \left(u_{1,2}^{l}+u_{2,1}^{l}-2u_{1,1}^{l}\right) \\[7pt]
   \displaystyle  u_{M,1}^{l+1} = u_{M,1}^{l}+\frac{\triangle_t}{h^2} \left(u_{M-1,1}^{l}+u_{M,2}^{l}-2u_{M,1}^{l}\right) \\[7pt]
   \displaystyle  u_{M,N}^{l+1} = u_{M,N}^{l}+\frac{\triangle_t}{h^2} \left(u_{M-1,N}^{l}+u_{M,N-1}^{l}-2u_{M,N}^{l}\right) \\[7pt]
    \displaystyle  u_{1,N}^{l+1} = u_{1,N}^{l}+\frac{\triangle_t}{h^2} \left(u_{2,N}^{l}+u_{1,N-1}^{l}-2u_{1,N}^{l}\right) \\[7pt]
    \displaystyle  u_{i,1}^{l+1} = u_{i,1}^{l}+\frac{\triangle_t}{h^2} \left(u_{i+1,1}^{l}+u_{i-1,1}^{l}+u_{i,2}^{l}-3u_{i,1}^{l}\right), \quad i=2, \cdots, M-1 \\[7pt]
    \displaystyle  u_{M,j}^{l+1} = u_{M,j}^{l}+\frac{\triangle_t}{h^2} \left(u_{M-1,j}^{l}+u_{M,j+1}^{l}+u_{M,j-1}^{l}-3u_{M,j}^{l}\right), \quad j=2, \cdots, N-1 \\[7pt]
    \displaystyle  u_{1,j}^{l+1} = u_{1,j}^{l}+\frac{\triangle_t}{h^2} \left(u_{2,j}^{l}+u_{1,j+1}^{l}+u_{1,j-1}^{l}-3u_{1,j}^{l}\right), \quad j=2, \cdots, N-1 \\[7pt]
    \displaystyle    u_{i,j}^{0} = u_{ij0} , \quad i=1, \cdots, M, \quad j=1, \cdots, N 
     \end{cases}
    \end{align}
    for $l>0$ and, for the nonlocal part,
    \begin{align}\label{discret2}
\begin{cases}
    \displaystyle  V_{k}^{l+1} = V_{k}^{l}+ \triangle_t R_2 h \sum_{p=1}^{N}J(z_k-z_p)(V_p^{l}-V_{k}^{l}) - \triangle_t h^2 \sum_{i=2}^{M-1}\sum_{j=2}^{N-1}G(x_i-z_k,y_j)(V_{k}^{l}-|R_2|u_{i,j}^{l}),\quad k=1, \cdots, M \\[7pt]
    \displaystyle    V_{k}^{0}  = {V}_{k0}, \quad k=1, \cdots, M,
    \end{cases}
    \end{align}
   for $l>0$. 

Similarly, the full discretization for the problem \eqref{limit-prob-neumann/robin} is given by: for the local part
\begin{align}\label{discret3}
\begin{cases}
 \displaystyle   u_{i,j}^{l+1} = u_{i,j}^{l}+\frac{\triangle_t}{h^2} \left(u_{i+1,j}^{l}+u_{i-1,j}^{l}+u_{i,j+1}^{l}+u_{i,j-1}^{l}-4u_{i,j}^{l}\right) \\[7pt]
 \displaystyle  u_{1,1}^{l+1} = u_{1,1}^{l}+\frac{\triangle_t}{h^2} \left(u_{1,2}^{l}+u_{2,1}^{l}-2u_{1,1}^{l}\right) \\[7pt]
   \displaystyle  u_{M,1}^{l+1} = u_{M,1}^{l}+\frac{\triangle_t}{h^2} \left(u_{M-1,1}^{l}+u_{M,2}^{l}-2u_{M,1}^{l}\right) \\[7pt]
   \displaystyle  u_{M,N}^{l+1} = u_{M,N}^{l}+\frac{\triangle_t}{h^2} \left(u_{M-1,N}^{l}+u_{M,N-1}^{l}-2u_{M,N}^{l}+h^2\sum_{k=1}^{M}G(x_M-z_k,y_N)(V_{k}^{l}-|R_2|u_{M,N}^{l})\right) \\[7pt]
    \displaystyle  u_{1,N}^{l+1} = u_{1,N}^{l}+\frac{\triangle_t}{h^2} \left(u_{2,N}^{l}+u_{1,N-1}^{l}-2u_{1,N}^{l}+h^2\sum_{k=1}^{M}G(x_1-z_{k},y_N)(V_{k}^{l}-|R_2|u_{1,N}^{l})\right) \\[7pt]
    \displaystyle  u_{i,1}^{l+1} = u_{i,1}^{l}+\frac{\triangle_t}{h^2} \left(u_{i+1,1}^{l}+u_{i-1,1}^{l}+u_{i,2}^{l}-3u_{i,1}^{l}\right), \quad i=2, \cdots, M-1 \\[7pt]
    \displaystyle  u_{i,N}^{l+1} = u_{i,N}^{l}+\frac{\triangle_t}{h^2} \left(u_{i+1,N}^{l}+u_{i-1,N}^{l}+u_{i,N-1}^{l}-3u_{i,N}^{l}+h^2\sum_{k=1}^{M}G(x_i-z_k,y_N)(V_{k}^{l}-|R_2|u_{i,N}^{l})\right), \\[7pt]
     \qquad   \qquad   \qquad   \qquad   \qquad i=2, \cdots, M-1 \\[7pt]
    \displaystyle  u_{M,j}^{l+1} = u_{M,j}^{l}+\frac{\triangle_t}{h^2} \left(u_{M-1,j}^{l}+u_{M,j+1}^{l}+u_{M,j-1}^{l}-3u_{M,j}^{l}\right), \quad j=2, \cdots, N-1 \\[7pt]
    \displaystyle  u_{1,j}^{l+1} = u_{1,j}^{l}+\frac{\triangle_t}{h^2} \left(u_{2,j}^{l}+u_{1,j+1}^{l}+u_{1,j-1}^{l}-3u_{1,j}^{l}\right), \quad j=2, \cdots, N-1,\\[7pt]
     \displaystyle    u_{i,j}^{0} = u_{ij0} , \quad i=1, \cdots, M, \quad j=1, \cdots, N,
     \end{cases}
    \end{align} 
    for $l>0$ and, for the nonlocal part 
    \begin{align}\label{discret4}
\begin{cases}
    \displaystyle  V_{k}^{l+1} = V_{k}^{l}+ \triangle_t R_2 h \sum_{p=1}^{N}J(z_k-z_p)(V_p^{l}-V_{k}^{l}) - \triangle_t h \sum_{i=1}^{M}G(x_i-z_k,y_N)(V_{k}^{l}-|R_2|u_{i,N}^{l})\quad k=1, \cdots, M \\[7pt]
   \displaystyle    V_{k}^{0}  = {V}_{k0}, \quad k=1, \cdots, M,
    \end{cases}
    \end{align} 
    for $l>0$.
    
    Notice that the main difference between the two discretizations occurs at the coupling terms, that in one case
    are given by 
    $$
    \sum_{k=1}^{M}G(x_i-z_{k},y_j)\left(V_{k}^{l}-|R_2|u_{i,j}^{l}\right)
    \qquad \mbox{and} \qquad
    \sum_{i=2}^{M-1}\sum_{j=2}^{N-1}G(x_i-z_k,y_j)(V_{k}^{l}-|R_2|u_{i,j}^{l})
    $$
    (these terms appear in the discretization of the model coupled via source terms, the
    double sums corresponds to discretizations of double integrals)
    and in the second discretization by
    $$
   \sum_{k=1}^{M}G(x_i-z_{k},y_N)(V_{k}^{l}-|R_2|u_{1,N}^{l}), \qquad \mbox{and} \qquad
      \sum_{i=1}^{M}G(x_i-z_k,y_N)(V_{k}^{l}-|R_2|u_{i,N}^{l})
    $$
(this corresponds to coupling on the boundary, remark that the sums here are discretizations of 
one dimensional integrals). 

For the experiments we will consider the domain $\Omega = [-1,1] \times [-1,1] $, $R_1 = [1,3]$, $R_2 = [0,1]$ and a time step which satisfies $\triangle_{t} \leq \frac{h^{2}}{4}$ (this comes from stability considerations). 

At the simulations we will use the kernel $J$, given by the following probability density:
\begin{align}\label{J}
J(x) = 
\begin{cases}
	\displaystyle \frac{1}{2}cos(x), \quad if \quad |x| \leq \frac{\pi}{2}, \\
	\displaystyle 0, \quad otherwise.
\end{cases}
\end{align} 
This particular kernel $J$ satisfies the hypothesis described before, $J$ is a nonnegative continuous function, symmetric, with $J(0)>0$ and integrable.

\subsection{Numerical experiments for coupling via source terms.}
Now, we will include some numerical experiments considering the fully discrete scheme for the problem \eqref{limit-problem-neumann}
given by \eqref{discret1}--\eqref{discret2}.

In this case, 
concerning the kernel $G$, as the problem \eqref{limit-problem-neumann} allows that particles can jump directly inside the interior of $\Omega$, we will consider $G$ as a function given by %that is more concentrated in the interior of $\Omega$ while it vanishes on the boundary. Let us take,
% \begin{align*}
%G(x,y) = 
%\begin{cases}
%	\displaystyle 1, \quad if \quad (x,y)=(x_{i},y_{j}), \quad i=2, \cdots, M-1, \quad j=2, \cdots, N-1\\
%	\displaystyle 0, \quad otherwise.
%\end{cases}
%\end{align*} 
 \begin{align}\label{G}
G(x,y) = 
\begin{cases}
	\displaystyle \frac{1}{4}cos(x)cos(y), \quad if \quad |x,y| \leq \frac{\pi}{2}\\
	\displaystyle 0, \quad otherwise.
\end{cases}
\end{align} 
The kernel $G$ satisfies the hypothesis defined in the Introduction. 

\medskip

{\bf Numerical experiment 1.}
For this simulation we consider $M=N=11$, $ h = 0,2$, $\triangle_{t} = 0,005$, as initial conditions, we used $u_{0}(x,y) = 0$, $V_{0}(x) = 1$.
The mean value of the initial condition is $\approx 0,083$.

In Figure 3 we plot the evolution of the local and the nonlocal parts of the solution (for the local part we have 
depicted the solution $u(x,y,t)$ at three different time steps, as the same for the nonlocal part of the solution, $V(x,t)$, we can observe its evolution in four time steps). 
Both local and nonlocal parts of the solution converge towards the mean value of the numerical initial condition as $t$ increases. 

\begin{figure}[H]
\centering
\includegraphics[trim=6cm 2cm 0cm 0cm, width=16cm]{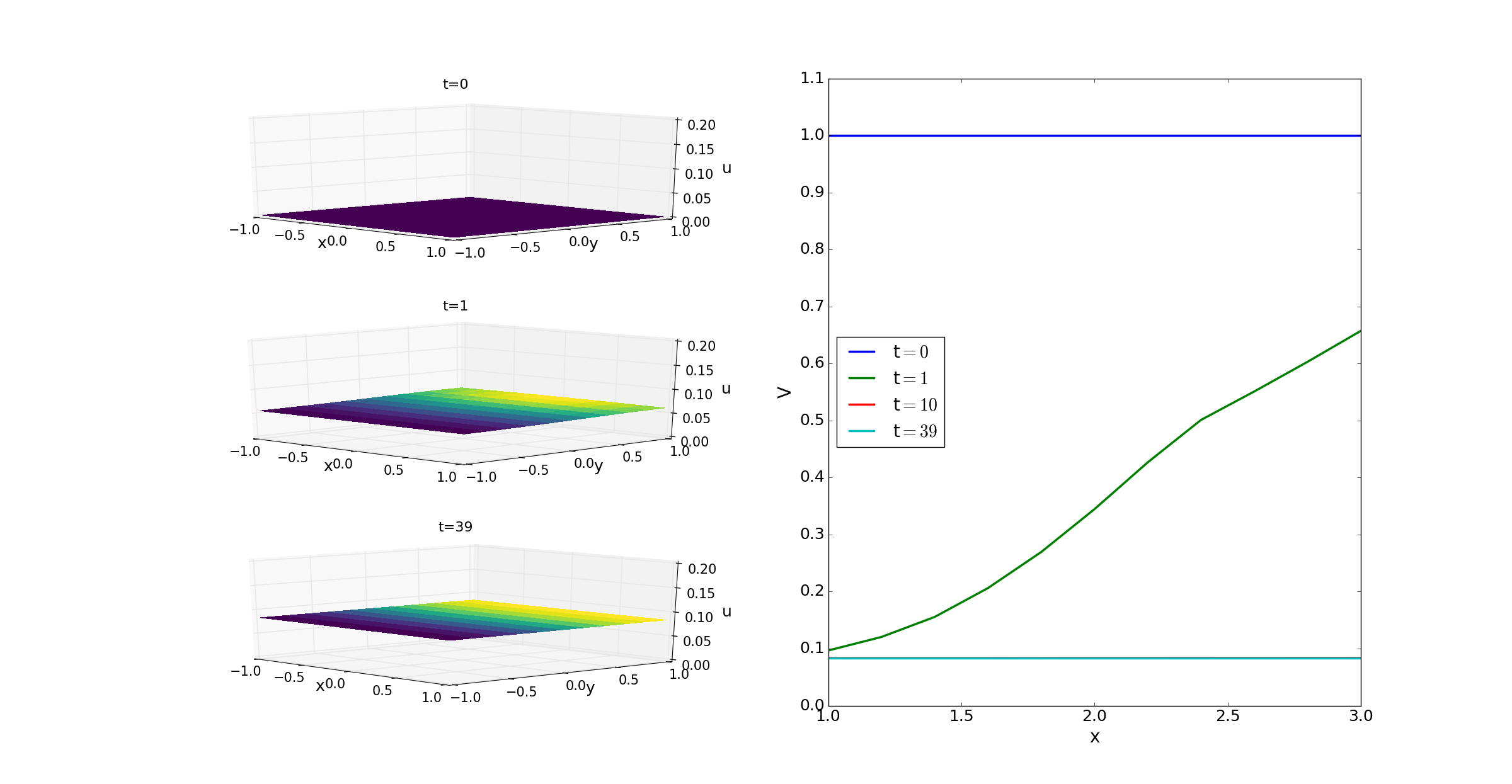}
\label{fig77}
\caption{The local part (left) and the nonlocal part (right) with two constants as initial conditions.}
\end{figure}

{\bf Numerical experiment 2.}
For this simulation we consider $M=N=11$, $ h = 0,2$, $\triangle_{t} = 0,005$, as initial conditions, we used $u_{0}(x,y) = cos\left(\frac{\pi x}{2}\right)cos\left(\frac{\pi y}{2}\right)$, $V_{0}(x) = 1$.
Now, the mean value of the initial condition $\approx 0,38$.

Figure 4 contains the plot of the local and the nonlocal parts of the solution. 
One can see that even with a not constant initial condition for the local part, we observe its fast convergence towards 
the mean value of the initial condition as $t$ increases. 

\begin{figure}[H]
\centering
\includegraphics[trim=6cm 2cm 0cm 0cm, width=16cm]{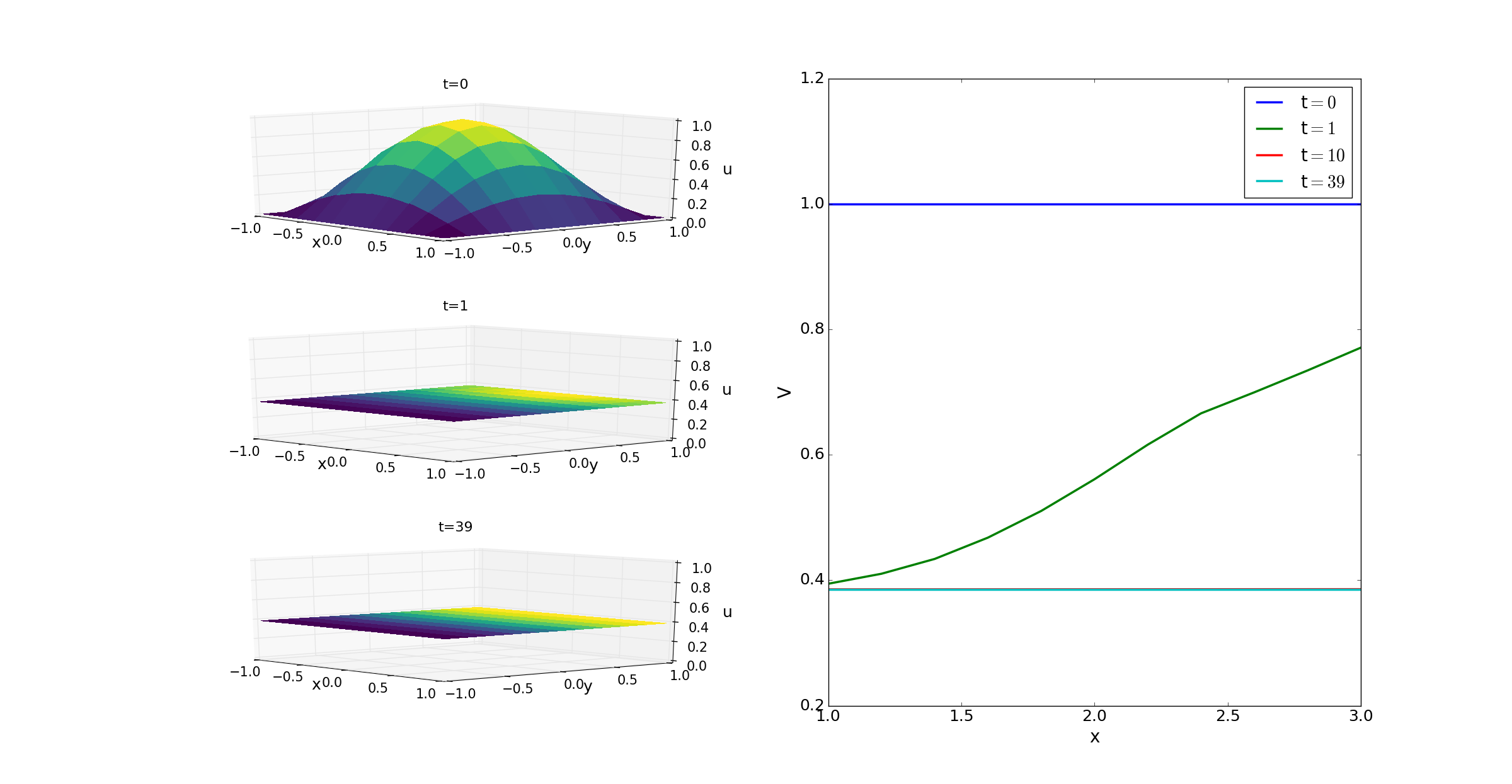}
\label{fig3}
\caption{The local part (left) and the nonlocal part (right) with non constant initial datum for the local part.}
\end{figure}

{\bf Numerical experiment 3.}
For this simulation we consider $M=N=11$, $ h = 0,2$, $\triangle_{t} = 0,005$, as initial conditions, we used $u_{0}(x,y) = cos\left(\frac{\pi x}{2}\right)cos\left(\frac{\pi y}{2}\right)$, $V_{0}(x) = 9-x^2$.
The mean value of the initial condition is $\approx 0,68$.

In Figure 5 both local and nonlocal initial conditions are non constants and they also 
verify the convergence to the mean of the initial condition as $t$ increases. Note that, even for $t=1$ the solution of the local part is closer to the mean value of the initial condition. For the nonlocal part, as $t=10$ the solution is very close to the mean of the initial condition that is subscribed by the last iteration.

\begin{figure}[H]
\centering
\includegraphics[trim=6cm 2cm 0cm 0cm, width=16cm]{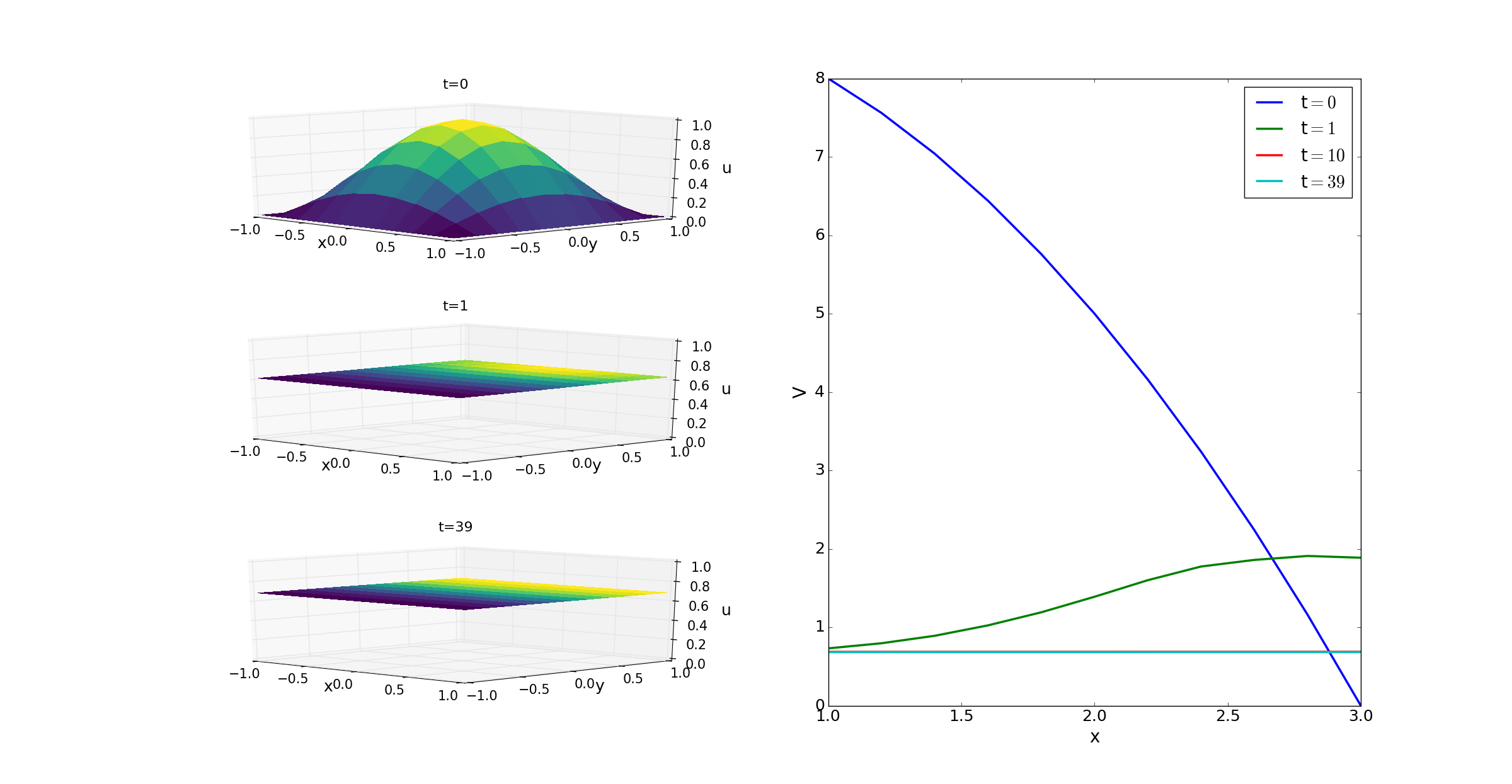}
\label{fig4}
\caption{The local part (left) and the nonlocal part (right) for non constant initial data.}
\end{figure}

\subsection{Numerical experiments for coupling via boundary terms.}
Now, we will include some numerical experiments considering the fully discrete scheme for the problem \eqref{limit-prob-neumann/robin} 
given by \eqref{discret3}--\eqref{discret4}. 

At the simulations we will use the same kernel $J$, as we define in \eqref{J} and the kernel $G$ as we define in \eqref{G}.

For simplicity, we have considered the local domain $\Omega$ as a square $\Omega =[-1,1]\times [-1,1]$, 
then for the coupling we will consider $\Gamma$ as a whole side of the domain $\Omega$, $\Gamma = \{1\} \times [-1,1]$. 
%We take as the kernel $G$:
% \begin{align*}
%G_1(x,y) = 
%\begin{cases}%
%	\displaystyle 1, \quad if \quad (x,y)=(x_{i},y_{N}), \quad i=1, \cdots, M \\
%	\displaystyle 0, \quad otherwise.
%\end{cases}
%\end{align*} 

{\bf Numerical experiment 4.}
For this simulation we consider $M=N=11$, $ h = 0,2$, $\triangle_{t} = 0,005$, as initial conditions, we used $u_{0}(x,y) = 0$, $V_{0}(x) = 1$.% and, in this case we used the kernel $G_1$.
Mean value of the initial condition $\approx 0,31$.

In Figure 6 we plot the evolution of the local and the nonlocal parts of the solution. Both local and nonlocal parts of the solution converge towards the mean value of the numerical initial condition as $t$ increases. %In this case, the kernel $G_1$ is concentrated in the whole side of the square $\Omega$, so in $t=1$ we can see the solution of the local part become increase, while the nonlocal solution becomes to decrease. 

\begin{figure}[H]
\centering
\includegraphics[trim=6cm 2cm 0cm 0cm, width=16cm]{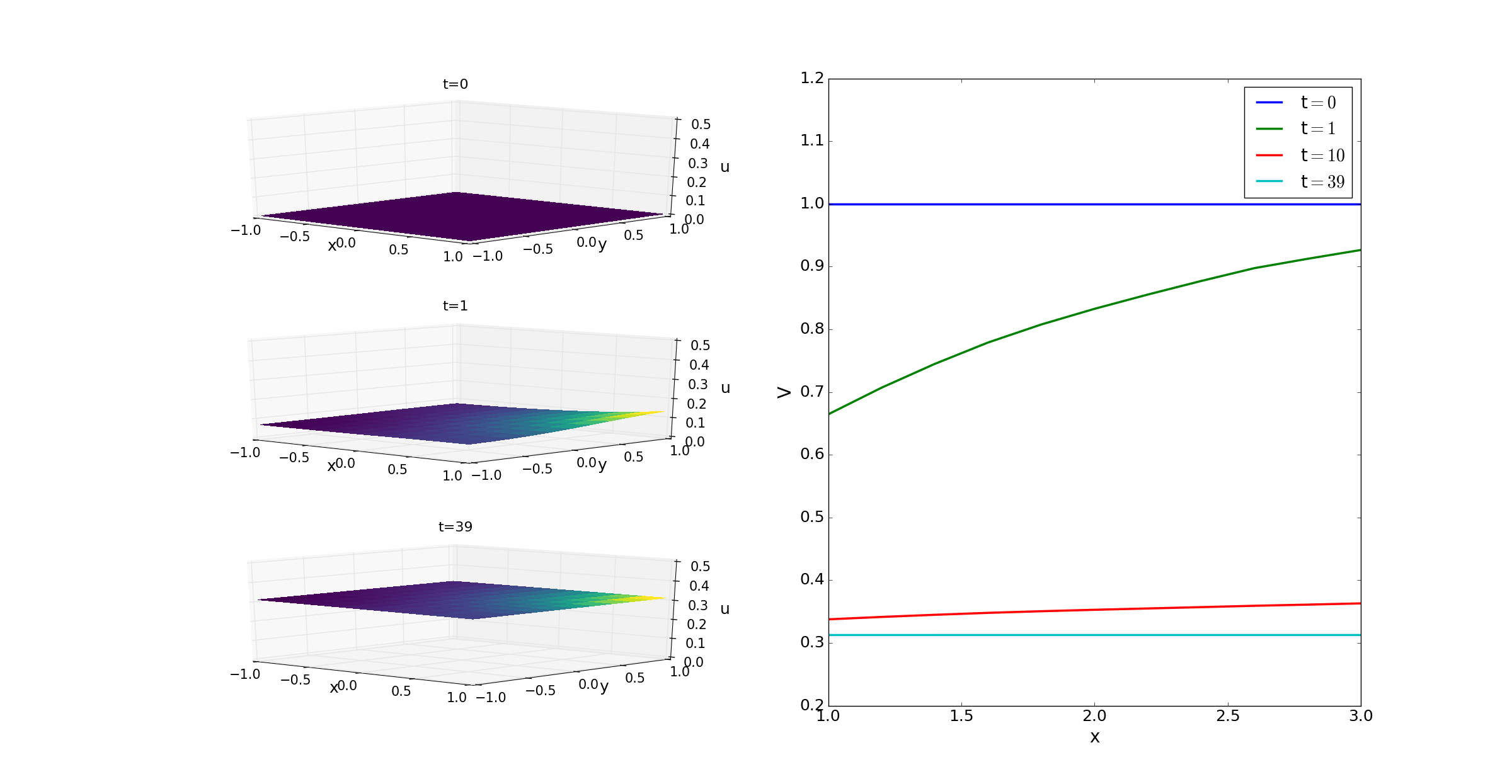}
\label{fig5}
\caption{The local part (left) and the nonlocal part (right).}
\end{figure}

{\bf Numerical experiment 5.}
For this simulation we consider $M=N=11$, $ h = 0,2$, $\triangle_{t} = 0,005$, as initial conditions, we used $u_{0}(x,y) = x^2+y^2$, $V_{0}(x) = 9-x^2$.% and, in this case we used the kernel $G_1$.
Mean value of the initial condition $\approx 1,99$.

In Figure 7, we observe that also when we take two non-constants initial conditions, the solution converges towards the mean value of the numerical initial condition as $t$ increases. We plot the solutions for specific time steps to follow the evolution.

\begin{figure}[H]
\centering
\includegraphics[trim=6cm 2cm 0cm 0cm, width=16cm]{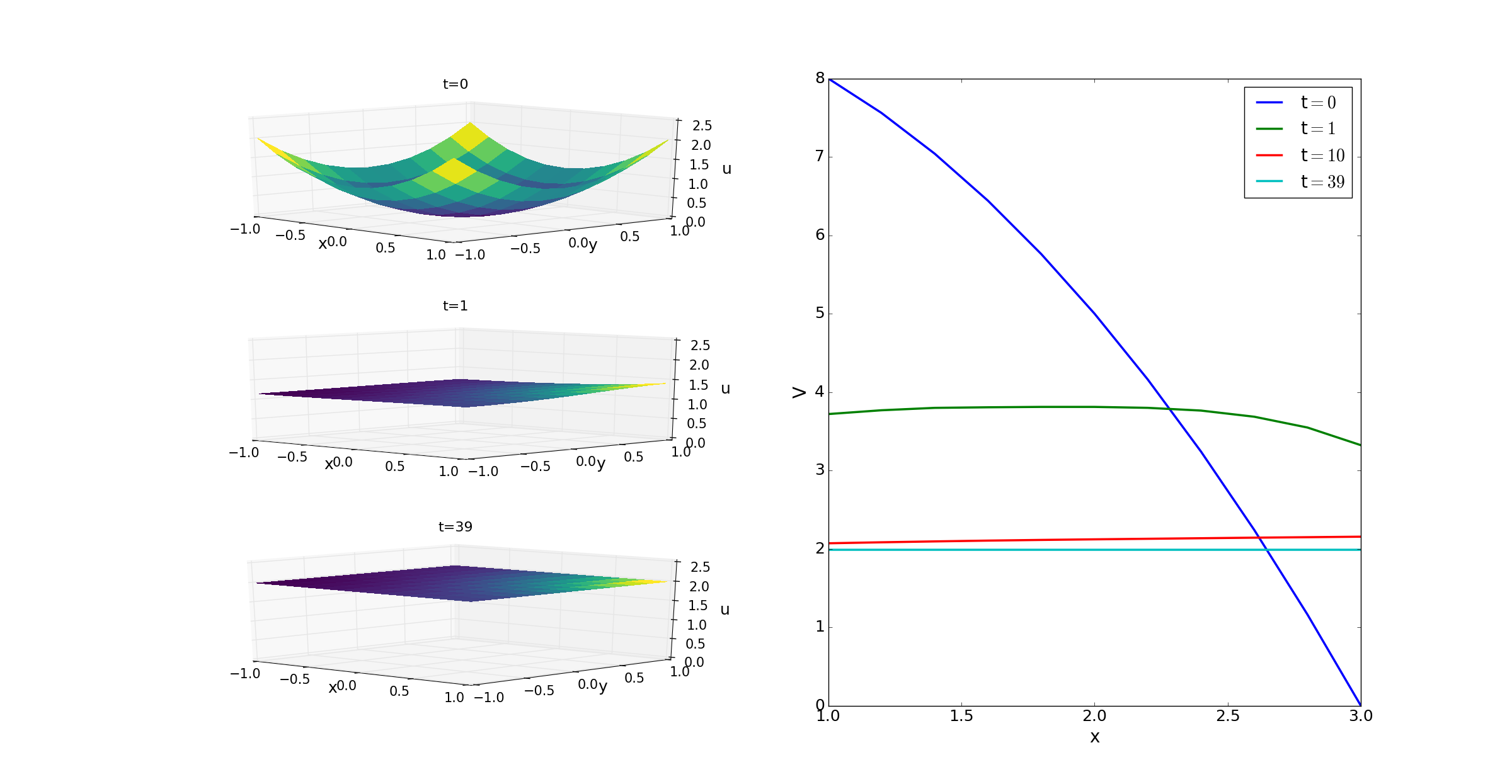}
\label{fig6}
\caption{The local part (left) and the nonlocal part (right).}
\end{figure}

{\bf Numerical experiment 6.}
For this simulation we consider $M=N=11$, $ h = 0,1$, $\triangle_{t} = 0,005$, as initial conditions, we used $u_{0}(x,y) = x^2+y^2$, $V_{0}(x) =x^2$.% and, in this case we used the kernel $G_1$.
Mean value of the initial condition $\approx 1,92$.
In Figure 8, we define the same initial condition for the local and nonlocal part. Note that we obtain the same behavior along the time, both local and nonlocal solution converge to the mean value of the initial condition. 
\begin{figure}[H]
\centering
\includegraphics[trim=6cm 2cm 0cm 0cm, width=16cm]{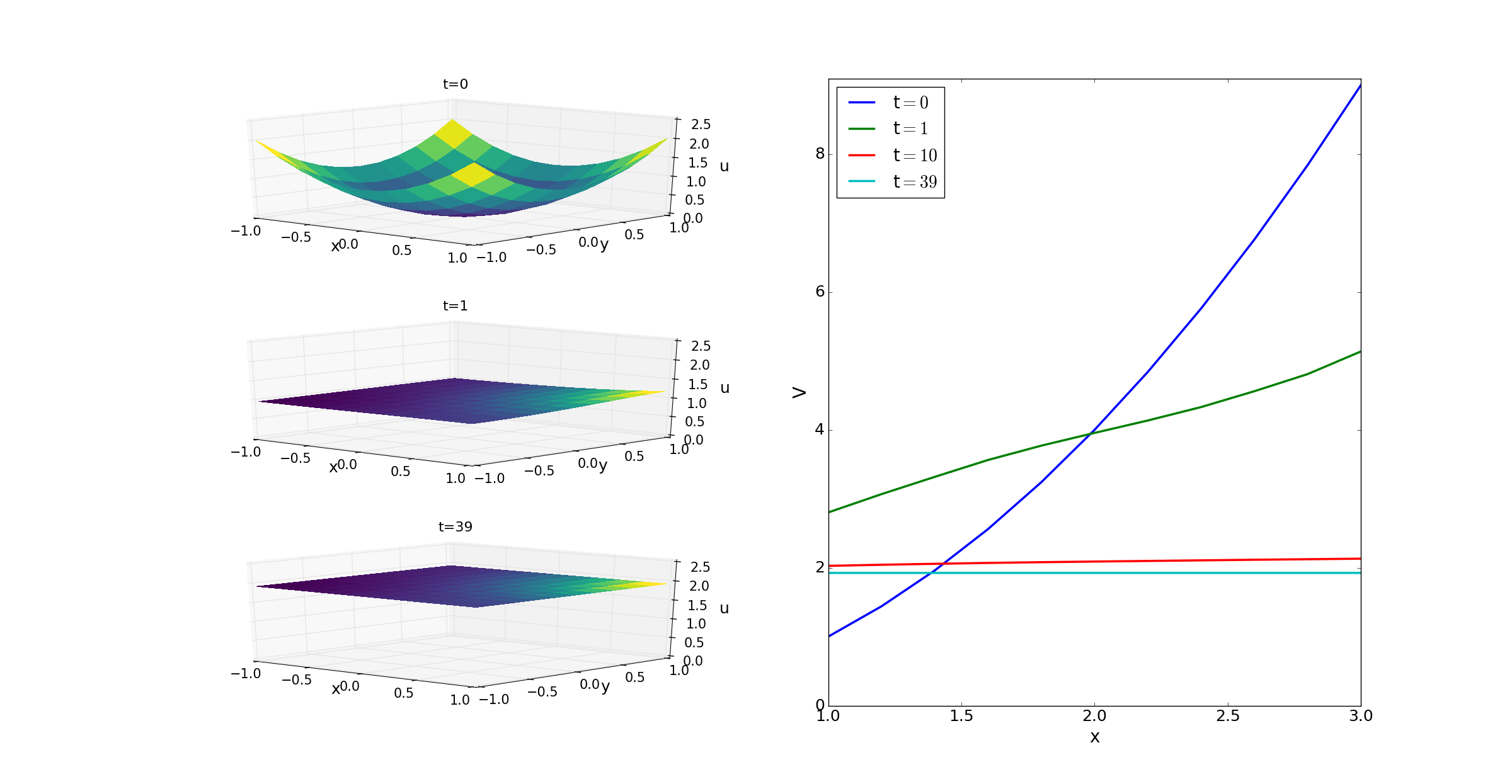}
\label{fig1}
\caption{Surface plot of the solution for the local part (left) and the nonlocal part (right).}
\end{figure}

\newpage

\medskip

%%%%%%%%%%%%%%%%%%%%%%%%%%%%%%%%%%%%%%%%%%%%%%%%%%%%%%%%%%%%%%%
\noindent{\large \textbf{Acknowledgments}}

BCS was financed by the Coordenação de Aperfeiçoamento de Pessoal de Nível Superior - Brasil (Capes) - No 88887369814/2019-00.

JDR is partially supported by CONICET grant PIP GI No 11220150100036CO
(Argentina), by  UBACyT grant 20020160100155BA (Argentina) and by the Spanish project MTM2015-70227-P.

%%%%%%%%%%%%%%%%%%%%

\bigskip

{\bf addresses}

B. C. dos Santos and S. Oliva. \\
IME-USP \\
  Institute of Mathematics and Statistics\\
  University of São Paulo, Brazil \\

J. D. Rossi \\ 
 Department of Mathematics, FCEyN\\
  University of Buenos Aires, Argentina

\end{document}